\documentclass[11pt, reqno]{amsart}
\usepackage{amssymb, hyperref}
\usepackage{amsthm}
\usepackage{amsmath}
\usepackage{graphicx}
\usepackage[all,2cell]{xy}
\usepackage{verbatim}
\usepackage{tikz}
\usepackage{tikz-cd}
\usepackage{hyperref}
\usepackage{array}
\usepackage[capitalise]{cleveref}
\usepackage{epsfig,graphicx}
\usepackage{tikz}
\usepackage{tikz-cd}
\usepackage{float}
\usepackage{capt-of}
\tikzset{node distance=2cm, auto}
\usepackage[vcentermath]{youngtab}
\usepackage{ytableau}
\usepackage{hyperref}
\usepackage{etoolbox}
\usetikzlibrary{graphs}
\usepackage{listings}
\usetikzlibrary{backgrounds}
\usepackage{multicol}
\usetikzlibrary{arrows,decorations.markings}
\tikzstyle{vertex}=[circle, draw, inner sep=0pt, minimum size=6pt]

\usepackage{setspace}
\usetikzlibrary{matrix,arrows,decorations.pathmorphing}
\usepackage{mathrsfs}
\usepackage{caption}
\numberwithin{equation}{section}
\newtheorem*{theorem*}{Theorem}
\newtheorem*{corollary*}{\bf Corollary}
\newtheorem*{remark*}{\bf Remark}
\usepackage{lipsum}
\usepackage{lineno}
\newtheorem{theorem}{Theorem}[section]
\newtheorem{corollary}[theorem]{Corollary}

\newtheorem{example}[theorem]{Example}
\newtheorem{lemma}[theorem]{Lemma}

\newtheorem{proposition}[theorem]{Proposition}

\newcommand{\eat}[1]{}

\title[torus quotients of Schubert varieties in the Grassmannian $G_{2, n}$]
{torus quotients of Schubert varieties in the Grassmannian $G_{2, n}$}
\author{S. Senthamarai Kannan}
\address{%
S. Senthamarai Kannan\\
Chennai Mathematical Institute \\
Plot H1, SIPCOT IT park,Siruseri,\\
 Kelambakkam-603103 \\
India\\
Email: kannan@cmi.ac.in\\
}
\author{Arpita Nayek}
\address{%
Arpita Nayek\\
Chennai Mathematical Institute \\
Plot H1, SIPCOT IT park,Siruseri,\\
Kelambakkam-603103 \\
India\\
Email: arpitan@cmi.ac.in\\
}
\author{Pinakinath Saha}
\address{%
Pinakinath Saha\\
Tata Inst. of Fundamental Research\\
Homi Bhabha Road, Colaba\\
Mumbai 400005\\
India\\
Email: psaha@math.tifr.res.in\\
}
\begin{document}
\maketitle
\begin{abstract}
Let $G=SL(n, \mathbb{C}),$ and $T$ be a maximal torus of $G,$ where $n$ is a positive even integer. In this article, we study the GIT quotients of the Schubert varieties in the Grassmannian $G_{2,n}.$ We prove that the GIT quotients of the Richardson varieties in the minimal dimensional Schubert variety admitting stable points in $G_{2,n}$ are projective spaces (see  Proposition \ref{lemma3.5} and Proposition \ref{prop3.6}). Further, we prove that the GIT quotients of certain Richardson varieties in $G_{2,n}$ are projective toric varieties. Also, we prove that the GIT quotients of the Schubert varieties in $G_{2,n}$ have at most finite set of singular points. Further, we have computed the exact number of singular points of the GIT quotient of $G_{2,n}.$ 
\end{abstract}


\section{Introduction}

Let $G=SL(n, \mathbb{C}).$  Let $T$ be a maximal torus of $G$, and $R$ be the set of roots with respect to $T.$ Let $R^{+}$ be a set of positive roots. Let $B$ (respectively, $B^{-}$) be the Borel subgroup of $G$ containing $T,$ corresponding to $R^{+}$ (respectively, $-R^{+}$). Let $S= \{\alpha_1,\ldots,\alpha_{n-1}\}$ denote the set of simple roots in $R^{+}$. Let  $P^{\alpha_{r}}$ denote the maximal parabolic subgroup of $G$ corresponding to the set $S\setminus\{\alpha_{r}\},$ for $1\le r\le n-1.$ Then  $G/P^{\alpha_{r}}$ is the Grassmannian parameterizing $r$-dimensional subspaces of $\mathbb{C}^{n}.$ We denote it by $G_{r,n}.$

For the action of  $T$ on the Grassmannian $G_{r,n},$ the GIT quotients have been studied extensively. In \cite{HK}, Hausmann and Knutson identified the GIT quotient of the Grassmannian  $G_{2,n}$ by $T$  with the moduli space of polygons in $\mathbb{R}^{3}.$ Also, they showed that GIT quotient of $G_{2,n}$ by $T$ can be realized as the GIT quotient of an $n$-fold product of projective lines by the diagonal action of $PSL(2, \mathbb{C}).$ In the symplectic geometry literature, these spaces are known as polygon spaces as they parameterize the $n$-sides polygons in $\mathbb{R}^{3}$ with fixed edge length up to rotation. More generally, GIT quotient of $G_{r,n}$ by $T$ can be identified with the GIT quotient of $(\mathbb{P}^{r-1})^{n}$ by the diagonal action of $PSL(r, \mathbb{C})$ called the Gel'fand-Macpherson correspondence. In \cite{Kap1}, \cite{Kap2}, Kapranov studied the Chow quotient of the Grassmannians and showed that Grothendieck-Knudsen moduli space $\overline{M_{0,n}}$ of stable $n$-pointed curves of genus zero arises as the Chow quotient of $T$ action on the Grassmannian $G_{2,n}.$ 

For a parabolic subgroup $Q$ of $G,$ Howard \cite{How} considered the problem of determining which line bundles on $G/Q$ descend to ample line bundles of the GIT quotient of $G/Q$ by $T.$ Howard showed that when $\mathcal{L}(\chi)$ is a very ample line bundle on $G/Q$ (so the character $\chi$ of $T$ extends to $Q$ and to no larger subgroup of $G$) and $H^{0}(G/Q,\mathcal{L}(\chi))^{T}$ is non-zero, the line bundle descends to the quotient \cite{How}. He extended the results to the case when the $T$-linearization of $\mathcal{L}(\chi)$ is twisted by $\psi,$ a character of $T.$ He proved that the line bundle $\mathcal{L}(\chi)$ twisted by $\psi$ descends to the GIT quotient provided the $\psi$-weight space of $H^{0}(G/Q, \mathcal{L}(\chi))$ is non-zero and this is so when $\chi-\psi$ is in the root lattice and $\psi$ is in the convex hull of the Weyl group orbit of $\chi.$ This was extended to any simple algebraic group by Kumar \cite{Kum}.   

In \cite{Sko}, Skorobogatov gave a combinatorial description using Hilbert-Mumford criterion, when a point in $G_{r,n}$ is semistable with respect to the $T$-linearized line bundle $\mathcal{L}(\omega_r).$  As a corollary he showed that when $r$ and $n$ are coprime semistablity is the same as stablility. Independently, in \cite{K1} and \cite{K2}, first named author gave a description of parabolic subgroups $Q$ of a simple algebraic group $G$ for which there exists an ample line bundle $\mathcal{L}$ on $G/Q$ such that $(G/Q)^{ss}_T(\mathcal{L})$ is the same as $(G/Q)^{s}_T(\mathcal{L}).$ In particular, in the case when $G=SL(n,\mathbb{C})$ and $Q=P^{\alpha_{r}},$ first named author showed that $(G_{r,n})^{ss}_T(\mathcal{L})$ is the same as $(G_{r,n})^{s}_T(\mathcal{L})$ if and only if $r$ and $n$ are co-prime.

In \cite{KS}, first named author and Sardar studied torus quotient of Schubert varieties in $G_{r,n}.$ They showed that $G_{r,n}$ has a unique minimal dimensional Schubert variety $X(w)$ admitting semistable points with respect to the $T$-lineararized bundle $\mathcal{L}(\omega_{r}),$ and gave a combinatorial characterization of the Weyl group element $w.$ In \cite{KP}, first named author and Pattanayak extended  the results to the case when $G$ is of type $B, C,$ or $D$ and when $P$ is a maximal parabolic subgroup of $G.$ 

In \cite{K1}, there was an attempt to study the projective normality of $T\backslash\backslash (G_{2,n})^{ss}_T(\mathcal{L}(n\omega_2))$ ($n$ is odd) with respect to the descent of the line bundle $\mathcal{L}(\omega_2).$ There it was proved that the homogeneous coordinate ring of $T\backslash\backslash G_{2,n}$ is a finite module over the subring generated by the degree one elements. In \cite{HMSV1}, Howard et al. showed that $T\backslash\backslash (G_{2,n})^{ss}_T(\mathcal{L}(\frac{n}{2}\omega_2))$ is projectively normal with respect to the descent of the line bundle $\mathcal{L}(\frac{n}{2}\omega_2)$ for $n$ is even, and $T\backslash\backslash (G_{2,n})^{ss}_T(\mathcal{L}(n\omega_2))$ is projectively normal with respect to the descent of the line bundle $\mathcal{L}(n\omega_2)$ for $n$ is odd. In \cite{BKS}, the authors use combinatorial method to show that $T\backslash\backslash (G_{2,n})^{ss}_T(\mathcal{L}(n\omega_2))$ is projectively normal with respect to the descent of the line bundle $\mathcal{L}(n\omega_2)$ for $n$ is odd.  In \cite{NP}, the authors use graph theoretic technique to give a short proof of the projective normality of $T\backslash\backslash (G_{2,n})^{ss}_T(\mathcal{L}(n\omega_2))$ for general $n$.

In \cite{HMSV2}, the authors studied the generators and relations for the GIT quotient of the Grassmannian $G_{2,n},$ $n$ is even, with respect to the action of $T$ using Gel'fand-Macpherson correspondence and they prove that the relations between the generators are quadratic except $n=6$. In this article, we study the GIT quotients of Schubert varieties in $G_{2,n}$ for $n$ even. We study the generators and relations of the homogeneous coordinate rings of the GIT quotients of the Schubert varieties in $G_{2,6}, G_{2,8},$ $G_{2,10}$ with respect to the action of $T$. We use  different method to give the explicit relations between them. 

Using Gel'fand-Macpherson correspondence, it is well-known that the GIT quotient of $G_{2,n}$ with respect to $T$ has finite number of singular points when $n$ is even  (see \cite[Example 8.8, p.158-159]{Mumford}). In this article, we study the singularities of the GIT quotients of the Schubert varieties in $G_{2,n}.$ We prove that they have at most finite number of singular points. Furthermore, using different method we show that the GIT quotient of $G_{2,n}$ has precisely half of the number of elements in  $W^{S\setminus\{\alpha_{\frac{n}{2}}\}}$ singular points.

The layout of the paper is as follows. In \cref{section2}, we recall some preliminaries on Algebraic groups, Lie algebras, Standard Monomial Theory, and Geometric Invariant Theory. In \cref{section3}, we prove that a Schubert variety $X(w)$ in $G/P$ admits stable point if and only if $w(a\chi)<0$ where $\chi$ is a  dominant character of $P$ such that $a\chi$ is in the root lattice (see \cref{lemma3.1}). By using this crucial lemma we find out the minimal dimensional Schubert variety admitting stable points in $G_{2,n}$ for $n$ even. Further, we prove that the GIT quotient of the minimal dimensional Schubert variety admitting stable points in $G_{2,n}$ is the projective space $\mathbb{P}^{\frac{n}{2}-1}$ (see Proposition \ref{lemma3.5}). Furthermore, we prove that the GIT quotients of the Richardson varieties in the minimal dimensional Schubert variety admitting stable point are projective spaces (see Proposition \ref{prop3.6}) and  the GIT quotients of certain Richardson varieties are toric varieties (see Corollary \ref{Cor 3.12}). In \cref{section4}, we study the structure of the GIT quotient of $G_{2,6}$. In \cref{section5}, and \cref{section6}, we study the structure of the GIT quotient of $X(6,8)$ in $G_{2,8}$ and $X(7,10)$ in $G_{2,10}$ respectively. In \cref{section7}, we prove that the GIT quotient of a Schubert variety $X(w)$ in $G_{2,n}$ has at most finite quotient singularities. Furthermore, we have shown that the GIT quotient of $G_{2,n}$ has precisely half of the number of elements in  $W^{S\setminus\{\alpha_{\frac{n}{2}}\}}$ singular points (see Corollary \ref{cor7.10}).      

\section{Notation and Preliminaries}\label{section2}
In this section, we set up some notation and preliminaries. We refer to \cite{Hum1}, \cite{Hum2}, \cite{Jan}, \cite{Mumford}, \cite{New}, \cite{LS} for preliminaries in Algebraic groups, Lie algebras, Standard Monomial Theory, and Geometric Invariant Theory.

Let $G, B,B^{-}, R,R^{+}, S, T,$ and $W$ be as in the introduction. The simple reflection in $W$ corresponding to $\alpha_i$ is denoted by $s_i$. For a subset $I$ of $S,$  we denote the parabolic subgroup of $G$ generated by $B$ and $\{n_{\alpha} : \alpha \in I\}$ by $P_{I},$ where $n_{\alpha}$ is a representative of $s_{\alpha}$ in $N_{G}(T).$  
We denote $P_{S\setminus I}$ by $P^{I.}$ In particular, $P_{S\backslash \{\alpha_r\}}$ is denoted by $P^{\alpha_r}.$ Note that all the standard maximal parabolic subgroups of $G$ are of the form $P^{\alpha_{r}}$ for some $1\le r\le n-1.$ Let $W_{P_{I}}$ be the subgroup of $W$ generated by $\{s_{\alpha}: \alpha \in I\}.$ We note that $W_{P_{I}}$ is the Weyl group of $P_{I}.$ For $I \subseteq S,$ $W^{P_{I}}=\{w \in W| w(\alpha) \in R^{+}, \text{ for all } \alpha \in I\}$ is the set of minimal coset representatives of elements of $W/W_{P_{I}}.$ Further, there is a natural order $W^{P_{I}},$ namely the restriction of the Bruhat order on $W.$ 
Let $I(r,n)=\{(a_{1},a_{2},\ldots, a_{r}): 1\le a_{1}<a_{2}<\cdots <a_{r}\le n\}.$ There is a natural order on $I(r,n),$ given by $(a_{1},a_{2},\ldots, a_{r})\le (b_{1},b_{2},\ldots, b_{r})$ if and only if $a_{i} \leq b_{i}$ for all $1\le i\le r.$ Then there is a natural order preserving indentification of $W^{P^{\alpha_{r}}}$ with $I(r,n)$ the correspondence is given by $w\in W^{P^{\alpha_{r}}}$ mapping to $(w(1),w(2),\ldots, w(r)).$ 

Let $V=\mathbb{C}^{n}.$ Let $\{e_{1},e_{2},\ldots,e_{n}\}$ be the standard basis of $V.$ Let $V_{r}$ be the subspace of $V$ spanned by $\{e_{1},e_{2},\ldots, e_{r}\}.$ Let $G_{r,n}$ be the set of all $r$-dimensional subspaces of $V.$ Then there is a natural projective variety structure on $G_{r,n}$ given by the Pl\"ucker embedding $\pi: G_{r,n}\longrightarrow \mathbb{P}(\wedge^{r}V)$ sending $r$-dimensional subspace to its $r$-th exterior power. Furthermore, there is a $G$-equivariant isomorphism from $G/P^{\alpha_{r}}$ to $G_{r,n},$ sending $gP^{\alpha_{r}}$ to $gV_{r}.$ For the natural action of $T$ on $G/P^{\alpha_{r}},$ the $T$-fixed points of $G/P^{\alpha_{r}}$ are identified with $W^{P^{\alpha_{r}}},$ sending $w$ to $wP^{\alpha_{r}}.$ For $w\in W^{P^{\alpha_r}},$ the closure of the $B$-orbit passing through $wP^{\alpha_{r}}$ in $G/P^{\alpha_{r}}$ has a natural structure of a projective variety called Schubert variety associated to $w,$ denoted by $X(w).$ Let $e_{\underline{i}}=e_{i_1} \wedge e_{i_2} \wedge \cdots \wedge e_{i_r},$ for $\underline{i}=(i_1, i_2, \ldots, i_r) \in I(r,n).$ Then $\{e_{\underline{i}}: \underline{i} \in I(r,n)\}$ forms a basis of $\wedge^rV.$ Let $\{p_{\underline{i}}: \underline{i}\in I(r,n)\}$ be the dual basis of the basis $\{e_{\underline{i}}: \underline{i} \in I(r,n)\}.$ Then $p_{\underline{i}}$'s are called the Pl\"{u}cker coordinates. Note that $p_{\underline{i}}|_{X(w)} \neq 0$ if and only if $\underline{i} \leq w.$ The monomials $p_{\tau_1}p_{\tau_2}\ldots p_{\tau_m} \in H^0(X(w),\mathcal{L}(\omega_r)^{\otimes m}),$ where $\tau_1, \tau_2, \ldots, \tau_m \in I_{r,n}$ are said to be standard monomial of degree $m$ if $\tau_1 \leq \tau_2 \leq \cdots \leq \tau_m \leq w.$ The standard monomials of degree $m$ on $X(w)$ form a basis of $H^0(X(w),\mathcal{L}(\omega_r)^{\otimes m})$ (see \cite[Theorem 4.5.0.5, p. 43]{LS}). 

Let $v,w\in W^{P^{\alpha_{r}}}$ and let $X^{v}:=\overline{B^{-}vP^{\alpha_{r}}}/P^{\alpha_{r}}$ be the opposite Schubert variety in $G/P^{\alpha_{r}}$ corresponding to $v$. Then the intersection of $X_{w}$ and $X^{v}$ in $G/P^{\alpha_{r}}$ with a reduced variety structure is called a Richardson variety in $G/P^{\alpha_{r}}$ and it is denoted by $X_{w}^{v}.$ Richardson varieties  are  
$T$-stable irreducible varieties. Further, $X^{v}_{w}$ is non empty if and only if $v\le w$ (see \cite[Lemma 1, p.655]{BL}).

Let $\mathfrak{g}$ be the Lie algebra of $G.$ Let $\mathfrak{h} \subset \mathfrak{g}$ be the Lie algebra of $T$ and $\mathfrak{b}$ be the Lie algebra of $B$. Let $X(T)$ (respectively, $Y(T)$) denote the group of all characters (respectively, one-parameter subgroups) of $T.$  Let $E_1:=X(T) \otimes \mathbb{R},$ $E_2:=Y(T) \otimes \mathbb{R}.$ Let $\langle . , . \rangle: E_1 \times E_2 \longrightarrow \mathbb{R}$ be the canonical non-degenerate bilinear form. Let $\bar{C}:=\{\lambda \in E_2 : \langle \alpha, \lambda \rangle \geq 0, \text{ for all } \alpha \in R^{+}\}.$ Note that for each $\alpha \in R,$ there is a homomorphism $\phi_{\alpha}:SL(2,\mathbb{C}) \longrightarrow G$, we have $\check{\alpha}:\mathbb{G}_m \longrightarrow G$ defined by $\check{\alpha}(t)=\phi_{\alpha}\bigg(\begin{pmatrix}
t & 0\\
0 & t^{-1}
\end{pmatrix}\bigg).$ We also have $s_{\alpha}(\chi)=\chi - \langle \chi, \check{\alpha} \rangle \alpha$ for all $\alpha \in R$ and $\chi \in E_1.$ Let $\{\omega_i: i= 1,2, \ldots, n-1\} \subset E_1$ be the fundamental weights; i.e. $\langle \omega_i, \check{\alpha_j} \rangle=\delta_{ij}$ for all $i,j=1, 2, \ldots, n-1.$ There is a natural partial order $\leq$ on $X(T)$ defined by $\psi\le \chi $ if and only if $\chi -\psi$ is a non-negative integral linear combination of simple roots.

Let $\mathcal{L}$ be a $T$-linealized ample line bundle on $G/P.$ We also denote the restriction of the line bundle $\mathcal{L}$ on $X(w)$ by $\mathcal{L}.$ A point $p \in X(w)$ is said to be a semistable with respect to the $T$-linearized line bundle $\mathcal{L}$ if there is a $T$-invariant section $s \in H^0(X(w),\mathcal{L}^{\otimes m})$ for some positive integer $m$ such that $s(p)\neq 0.$ We denote the set of all semistable points of $X(w)$ with respect to $\mathcal{L}$ by $X(w)^{ss}_T(\mathcal{L}).$ A point $p$ in $X(w)^{ss}_{T}(\mathcal{L})$ is said to be a stable point if $T$-orbit of $p$ is closed in $X(w)^{ss}_{T}(\mathcal{L})$ and stabilizer of $p$ in $T$ is finite. We denote the set of all stable points of $X(w)$ with respect to $\mathcal{L}$ by $X(w)^{s}_T(\mathcal{L}).$ 

Now we recall the definition of projective normality of a projective variety.
Let $X$ be a projective variety in $\mathbb{P}^{m}.$ We denote by $\hat{X}$ the affine cone of $X.$ $X$ is said to be projectively normal if $\hat{X}$ is normal. For a reference, see exercise 3.18, page. 23 of \cite{R}. For the practicle purpose we need the following fact about projective normality of a polarized variety.

A polarized variety $(X,\mathcal{L})$ where $\mathcal{L}$ is a very-ample line bundle is said to be projectively normal if its homogeneous coordinate ring $\bigoplus\limits_{k\ge 0} H^{0}(X,\mathcal{L}^{\otimes k})$ is integrally closed and is generated as a $\mathbb{C}$-algebra by $H^0(X,\mathcal{L})$ (see exercise 5.14, chapter II of \cite{R}).

\section{GIT quotient of the minimal dimensional Schubert variety admitting stable points}\label{section3}
In this section, we prove a crucial lemma on the existence of stable points of a Schubert variety $X(w)$ in $G/P.$
\begin{lemma}\label{lemma3.1}
Let $G$ be a simple algebraic group of rank $n.$ Let $\chi$ be a non-trivial dominant character of $T$. Let $\chi=\sum_{i=1}^nm_i\omega_i$. Let $I=\{\alpha_i \in S: m_i \geq 1\}$. Let $P=P^{I}$. Let $a \in \mathbb{N}$ be such that $a\chi$ is in root lattice. Let $w \in W^{P}.$ Then $w(a\chi)=\sum_{i=1}^na_i\alpha_i$ with $a_i \leq -1$ for all $1 \leq i \leq n$ if and only if $X(w)^s_T(\mathcal{L}(\chi)) \neq \emptyset.$  
\end{lemma}
\begin{proof}
Let $Z$ be the union of the $W$-translates of $X(v)$, $v \in W^P$ with $v \ngeq w$. Since $X(w)$ is irreducible, by using the argument of \cite[Lemma 2.1, p.~470]{KP} we have $X(w) \nsubseteq Z$. Let $x \in X(w) \backslash Z$. By construction for any $\tau \in W$, $\tau x \in BvP/P$ implies $v \geq w$.

Now, let $\lambda$ be a nontrivial one parameter subgroup of $T$. Choose $\tau \in W$ such that $\tau(\lambda)=\sum_{i=1}^nc_i\lambda_i$ with each $c_i \geq 0$ and there is an $1 \leq i_0 \leq n$ such that $c_{i_0} \geq 1$. 

By the above discussion, $\tau x \in BvP/P$ for some $v \in W^P$ such that $v \geq w$.

Hence, by\cite[Lemma 5.1, p.534]{Ses}\begin{multline} \mu^{\mathcal{L}(a\chi)}(x, \lambda) = \mu^{\mathcal{L}(a\chi)}(\tau x, \tau \lambda)
=-\langle v(a \chi), \sum_{i=1}^n c_i \lambda_i \rangle \geq -\langle w(a \chi), \sum_{i=1}^n c_i \lambda_i \rangle = -\langle \sum_{i=1}^na_i\alpha_i, \sum_{i=1}^n c_i \lambda_i \rangle \\= -\sum_{i=1}^na_ic_i \geq 1.\end{multline}

Therefore, by Hilbert-Mumford criterion, we have $x \in X(w)^s_T(\mathcal{L}(\chi))$.

Conversely, suppose $X(w)^{s}_{T}(\mathcal{L}(a\chi))\neq \emptyset.$

Then $X(w)^{s}_{T}(\mathcal{L}(a\chi))\cap BwP/P\neq \emptyset.$ Let $x\in X(w)^{s}_{T}(\mathcal{L}(a\chi))\cap BwP/P.$ Let $\lambda$ be a non-trivial one parameter subgroup of $T.$ Since $x$ is stable point, by using Hilbert-Mumford criterion we have $\mu^{\mathcal{L}(a\chi)}(x, \lambda)>0.$ In particular, $\mu^{\mathcal{L}(a\chi)}(x, \lambda_{i})>0$ for all $1\le i\le n.$  Since $x\in BwP/P$ and $\lambda_{i}$'s are fundamental one parameter subgroups of $T,$ by \cite[Lemma 5.1, p.534]{Ses} we have $\mu^{\mathcal{L}(a\chi)}(x, \lambda_{i})=-\langle w(a\chi),\lambda_{i} \rangle>0.$ That is $-\langle \sum\limits_{j=1}^{n}a_{j}\alpha_{j}, \lambda_{i}\rangle>0$ for all $1\le i\le n.$ Therefore, $a_{i}\le-1$ for all $1\le i\le n.$  
\end{proof}
 
We use Lemma \ref{lemma3.1} to find out the minimal dimensional Schubert variety admitting stable points in $G_{2,n}$ for $n$ even, and study its GIT quotient. To proceed further, we recall the following theorem as it is in \cite{HMSV1}. The authors consider the space of weighted points on $\mathbb{P}^n$. Let the $i$-th point be weighted by $a_i$, and let $\underline{a}=(a_1, \ldots, a_n) \in (\mathbb{Z}^{+})^n$ (the weight vector). The weights can be interpreted as parameterizing the very ample line bundles of $(\mathbb{P}^1)^n$. Let $R_{\underline{a}}$ be the homogeneous coordinate ring of $(\mathbb{P}^1)^n//PSL(2,\mathbb{C})$ given by the descent of the very-ample line bundle corresponding to $\underline{a}.$

\begin{theorem}[\cite{HMSV1},Theorem 2.3, p.~182]\label{theo3.2}
The lowest degree invariants generate the coordinate ring $R_{\underline{a}}.$
\end{theorem}

\begin{corollary}\label{cor3.3}
  The GIT quotient $T \backslash\backslash (G_{2,n})^{ss}_T(\mathcal{L}({\frac{n}{2}\omega_2}))$ is projectively normal with respect to the descent of the line bundle $\mathcal{L}({\frac{n}{2}\omega_2})$.
\end{corollary}
\begin{proof}
	By Gel'fand-MacPherson correspondence there is an isomorphism of 
	$((\mathbb{P}^1)^n//PSL(2,\mathbb{C}),\\ \mathcal{L}(1,1,\ldots,1))$ with $(T \backslash\backslash G_{2,n}, \mathcal{L}(\frac{n}{2}\omega_2))$ as a polarized variety. Therefore, by \cref{theo3.2}, corollary follows.
\end{proof}

\begin{corollary}\label{cor3.4}
 The GIT quotients of  Richardson varieties in $G_{2,n}$ by $T$ are projectively normal with respect to the descent of the line bundles $\mathcal L(\frac{n}{2}\omega_2)$.
 \end{corollary}
 \begin{proof} 
Let  $v, w \in W^{P^{\alpha_2}}$ be such that $v\le w.$ By \cite[Lemma 5, Proposition 1, p. 658]{BL} the restriction map $\phi: H^0(G_{2,n}, \mathcal{L}(\frac{n}{2}\omega_2)^{\otimes d}) \to H^0(X^v_w, \mathcal{L}(\frac{n}{2}\omega_2)^{\otimes d}))$ is surjective. Further, since $T$ is linearly reductive, the restriction map $\phi: H^0(G_{2,n}, \mathcal{L}(\frac{n}{2}\omega_2)^{\otimes d})^T \to H^0(X^v_w, \mathcal{L}(\frac{n}{2}\omega_2^{\otimes d})^T$ is surjective. So, by \cref{cor3.3}, $\bigoplus\limits_{d \geq 0}H^0(X^v_w, \mathcal{L}(\frac{n}{2}\omega_2)^{\otimes d})^T$ is generated by $H^0(X^v_w, \mathcal{L}(\frac{n}{2}\omega_2))^T$. Hence, $T \backslash \backslash(X^v_w)^{ss}_T(\mathcal{L}(\frac{n}{2}\omega_2))$ is projectively normal.
 \end{proof}

\subsection{Minimal dimensional Schubert variety admitting stable points} 
Let $w^{ss}_{\text{min}}\\=(s_{\frac{n}{2}-1}\ldots s_1)(s_{n-1}s_{n-2}\ldots s_2)=(\frac{n}{2},n).$ Then $w^{ss}_{\text{min}}$ is a unique minimal element in $W^{S\setminus\{\alpha_{2}\}}$ such that $w^{ss}_{\text{min}}(\frac{n}{2}\omega_{2})\le 0.$ Therefore, by \cite[Lemma 2.1, p.~470]{KP}, $X(w^{ss}_{\text{min}})$ is the minimal dimensional Schubert variety in $G_{2,n}$ admitting semistable points.
 Note that $\prod_{k=1}^{\frac{n}{2}}p_{k,\frac{n}{2}+k}$ is the unique standard monomial in $H^0(X(w^{ss}_{\text{min}}),\mathcal{L}(\frac{n}{2}\omega_2))^T$. Therefore, by \cref{cor3.4}, $T \backslash\backslash (X(w^{ss}_{\text{min}})^{ss}_T(\mathcal{L}(\frac{n}{2}\omega_2))$ is a point.

Let $w^s_{\text{min}}=(s_{\frac{n}{2}}s_{\frac{n}{2}-1}\ldots s_1)(s_{n-1}s_{n-2}\ldots s_2)=(\frac{n}{2}+1,n)$. Note that $l(w^s_{\text{min}})=\frac{n}{2}-2+n.$ Then $w^{s}_{\text{min}}$ is a unique minimal element in $W^{S\setminus\{\alpha_{2}\}}$ such that $w^{s}_{\text{min}}(\frac{n}{2}\omega_{2})< 0.$ Therefore, by \cref{lemma3.1}, $X(w^{s}_{\text{min}})$ is the minimal dimensional Schubert variety in $G_{2,n}$ admitting stable points. 
 Note that the standard monomials in $H^0(X(w^{s}_{\text{min}}),\mathcal{L}_{\frac{n}{2}\omega_2})^T$ are $X_t=p_{1,t}\prod\limits_{k=2}^{t-1}p_{k,\frac{n}{2}+k}\prod\limits_{k=t}^{\frac{n}{2}}p_{k+1,\frac{n}{2}+k}$, for all $2 \leq t \leq \frac{n}{2}+1$.
 
 \begin{proposition}\label{lemma3.5}
 The GIT quotient $T \backslash\backslash (X(w^s_{\text{min}}))^{ss}_T(\mathcal{L}(\frac{n}{2}\omega_2))$ is isomorphic to $\mathbb{P}^{\frac{n}{2}-1}.$
 \end{proposition}
 
 \begin{proof} 
 Let $R=\bigoplus\limits_{d\ge 0} R_{d},$ where $R_d=H^0(X(w^{s}_{\text{min}}), \mathcal{L}(\frac{dn}{2}\omega_2))^{T}.$ Then $T \backslash\backslash (X(w^s_{\text{min}}))^{ss}_T(\mathcal{L}(\frac{n}{2}\omega_2))=\text{Proj}(R).$ By \cref{cor3.4}, $R$ is generated as a $\mathbb{C}$-algebra by $X_t$'s, where $2 \leq t \leq \frac{n}{2}+1.$ 
 On the other hand, since $X(w^{s}_{\text{min}})$ admits stable points, we have $dim(T \backslash\backslash X(w^s_{\text{min}}))=\frac{n}{2}-1.$ Therefore, $R$ is the polynomial algebra generated by $X_t$'s, where $2 \leq t \leq \frac{n}{2}+1$. Hence, $Proj(R) \cong\mathbb{P}^{\frac{n}{2}-1}$.
    
 \end{proof}
\subsection{Torus quotient of Richardson varieties}
Next we generalize Proposition \ref{lemma3.5} to the Richardson varieties in $X(w_{min}^{s}).$ We prove the following
 \begin{proposition}\label{prop3.6}
Let $v_{k}=s_k s_{k-1} \cdots s_2$ for all $2 \leq k \leq \frac{n}{2}.$ Then the GIT quotient \\ $T\backslash \backslash (X^{v_k}_{w^{s}_{min}})^{ss}_T(\mathcal{L}(n\omega_2))$ is isomorphic to $\mathbb{P}^{\frac{n}{2}-k}.$
\end{proposition}
\begin{proof}
Let $R_{v_{k},w}=\bigoplus\limits_{d\ge 0} R_{d},$ where $R_d=H^0(X^{v_k}_{w^{s}_{\text{min}}}, \mathcal{L}(\frac{dn}{2}\omega_2))^{T}.$ Then $T \backslash\backslash (X^{v_k}_{w^{s}_{\text{min}}})^{ss}_T(\mathcal{L}(\frac{n}{2}\omega_2))=\text{Proj}(R_{v_{k},w}).$ Note that $X_{t}$'s  forms a basis of $R_{1}$ where $k+1\le t\le \frac{n}{2}+1.$ Further, we observe that $\prod\limits_{t=k+1}^{\frac{n}{2}+1} X_{t}^{a_{t}}$'s are standard monomials for all non-negative integers $a_{t}$ where $k+1\le t\le \frac{n}{2}+1.$  By \cref{cor3.4}, $R_{v_{k},w}$ is generated by $X_t$'s, where $k+1 \leq t \leq \frac{n}{2}+1$ as a $\mathbb{C}$-algebra. 
Therefore, $R_{v_{k},w}$ is the polynomial algebra generated by $X_t$'s, where $k+1 \leq t \leq \frac{n}{2}+1$. Hence, $Proj(R_{v_{k},w})$ is isomorphic to $\mathbb{P}^{\frac{n}{2}-k}.$
\end{proof}
Let $w=(s_{\frac{n}{2}+1}\cdots s_{1})(s_{n-1}\cdots s_{2})$ and $v_{k}=s_{k}s_{k-1}\ldots s_2$ for all $2 \leq k \leq \frac{n}{2}-1.$ Then in the one line notation we have $w=(\frac{n}{2}+2,n)$ and $v_k=(1,k+1)$ for all $2 \leq k \leq \frac{n}{2}-1.$

Let $Y_{i,j}=p_{1,i}p_{2,j}(\prod\limits_{l=3}^{i-1}p_{l,\frac{n}{2}+l})(\prod\limits_{l=i+1}^{j-1}p_{l,\frac{n}{2}+l-1})(\prod\limits_{l=j+1}^{\frac{n}{2}+2}p_{l,\frac{n}{2}+l-2})$ for all $k+1 \leq i<j \leq \frac{n}{2}+2.$ 

Let $R_{v_{k},w}=\bigoplus\limits_{d\geq 0}R_d$ where $R_d=H^{0}(X^{v_k}_{w}, \mathcal{L}(\frac{dn}{2}\omega_2))^T$ for all $2\le k\le \frac{n}{2}-1.$ Then we have

\begin{lemma}\label{lemma3.7}
The standard monomials of $R_1$ are in the form of $Y_{i,j}$ where $k+1 \leq i<j \leq \frac{n}{2}+2.$
\end{lemma}
\begin{proof}
	Any standard Young tableau associated to standard monomials in $R_1$ has $\frac{n}{2}$ rows and $2$ columns with strictly increasing rows and non-decreasing columns. Let $Row_i$ denotes the $i$-th row of the tableau for all $1 \leq i \leq \frac{n}{2}$ and $E_{i,j}$ be the $(i,j)$-th entry of the tableau. Since $Row_{\frac{n}{2}} \leq (\frac{n}{2}+2,n),$ we have $E_{\frac{n}{2},1} \leq \frac{n}{2}+2.$ Hence, all the integers between $\frac{n}{2}+3$ and $n$ appears in the second column. Hence, $E_{i,2}=\frac{n}{2}+i$ for all $3 \leq i \leq \frac{n}{2}.$ Further, since $Row_1 \geq (1, k+1),$ we have $E_{1,2} \geq k+1.$ Hence, all the integers between $1$ and $k$ appears in the first column. Hence, $E_{i,1}=i$ for all $1 \leq i \leq k.$ Hence, $k+1 \leq E_{1,2}<E_{2,2}\leq \frac{n}{2}+2.$  So, the standard monomials in $R_1$ are of the form   $Y_{i,j}.$
\end{proof}
Next our goal is to study the structure of the GIT quotient of $T\backslash \backslash (X^{v_{k}}_w)^{ss}_T(\mathcal{L}(\frac{n}{2}\omega_2))$ for all $2\le k\le \frac{n}{2}-1.$ 
First we consider $T\backslash \backslash (X^{v_{k}}_w)^{ss}_T(\mathcal{L}(\frac{n}{2}\omega_2))$ for $k=\frac{n}{2}-1.$  Then we have  
\begin{lemma}
The GIT quotient $T \backslash \backslash (X^{v_{\frac{n}{2}-1}}_w)^{ss}_T(\mathcal{L}(\frac{n}{2}\omega_2))$ is isomorphic to $\mathbb{P}^2.$
\end{lemma}
\begin{proof}
By \cref{lemma3.7}, the standard monomials of $R_1$ are of the form
\begin{enumerate}
\item[(i)] $Y_{\frac{n}{2}+1,\frac{n}{2}+2}=\prod\limits_{l=1}^{\frac{n}{2}}p_{l,\frac{n}{2}+l};$

\item[(ii)] $Y_{\frac{n}{2},\frac{n}{2}+2}=p_{1,\frac{n}2}(\prod\limits_{l=2}^{\frac{n}{2}-1}p_{l,\frac{n}{2}+l})p_{\frac{n}{2}+1,n};$

\item[(iii)] $Y_{\frac{n}{2},\frac{n}{2}+1}=p_{1,\frac{n}{2}}p_{2,\frac{n}{2}+1}(\prod\limits_{l=3}^{\frac{n}{2}-1}p_{l,\frac{n}{2}+l})p_{\frac{n}{2}+2,n}.$
\end{enumerate}

We observe that $Y_{\frac{n}{2},\frac{n}{2}+1}^a Y_{\frac{n}{2},\frac{n}{2}+2}^bY_{\frac{n}{2}+1,\frac{n}{2}+2}^c$'s are standard monomials for all non-negative integers $a,b,$ and $c.$  By \cref{cor3.4}, the homogeneous coordinate ring of $T \backslash \backslash (X^{v_{\frac{n}{2}-1}}_w)^{ss}_T(\mathcal{L}(\frac{n}{2}\omega_2))$ is generated by $Y_{\frac{n}{2},\frac{n}{2}+1},Y_{\frac{n}{2},\frac{n}{2}+2}$ and $Y_{\frac{n}{2}+1, \frac{n}{2}+2}.$ Therefore, $T \backslash \backslash (X^{v_{\frac{n}{2}-1}}_w)^{ss}_T(\mathcal{L}(\frac{n}{2}\omega_2))$  is isomorphic to $\mathbb{P}^2.$
\end{proof}
Next we consider $T\backslash \backslash (X^{v_{k}}_w)^{ss}_T(\mathcal{L}(\frac{n}{2}\omega_2))$ for all $2 \leq k \leq \frac{n}{2}-2.$
\begin{lemma}
The following relations hold in $R_2.$
\begin{itemize}
\item[(i)] $Y_{i,j}Y_{m,s}=Y_{i,m}Y_{j,s};$
		
\item[(ii)] $Y_{i,m}Y_{j,s}=Y_{i,s}Y_{j,m};$ 
\end{itemize}
for all $k+1 \leq i < j < m < s \leq \frac{n}{2}+2.$
\end{lemma}
\begin{proof}
Proof of $(i)$:

Recall that 

$Y_{i,j}=p_{1,i}p_{2,j}(\prod\limits_{l=3}^{i-1}p_{l,\frac{n}{2}+l})(\prod\limits_{l=i+1}^{j-1}p_{l,\frac{n}{2}+l-1})(\prod\limits_{l=j+1}^{\frac{n}{2}+2}p_{l,\frac{n}{2}+l-2})$ and

$Y_{m,s}=p_{1,m}p_{2,s}(\prod\limits_{l=3}^{m-1}p_{l,\frac{n}{2}+l})(\prod\limits_{l=m+1}^{s-1}p_{l,\frac{n}{2}+l-1})(\prod\limits_{l=s+1}^{\frac{n}{2}+2}p_{l,\frac{n}{2}+l-2})).$

Since $i< j< m< s,$ we write 

$Y_{i,j}=p_{1,i}p_{2,j}(\prod\limits_{l=3}^{i-1}p_{l,\frac{n}{2}+l})(\prod\limits_{l=i+1}^{j-1}p_{l,\frac{n}{2}+l-1})(\prod\limits_{l=j+1}^{m}p_{l,\frac{n}{2}+l-2})(\prod\limits_{l=m+1}^{\frac{n}{2}+2}p_{l,\frac{n}{2}+l-2})$ and

$Y_{m,s}=p_{1,m}p_{2,s}(\prod\limits_{l=3}^{j-1}p_{l,\frac{n}{2}+l})(\prod\limits_{l=j}^{m-1}p_{l,\frac{n}{2}+l}) (\prod\limits_{l=m+1}^{s-1}p_{l,\frac{n}{2}+l-1})(\prod\limits_{l=s+1}^{\frac{n}{2}+2}p_{l,\frac{n}{2}+l-2}).$

Note that since $ i< j < m < s ,$ $Y_{i,j}Y_{m,s}$ is not standard. We apply the straightening law repeatedly on the nonstandard factors of $Y_{i,j}Y_{m,s}$ to express $Y_{i,j}Y_{m,s}$ as linear combination of standard monomials.
Further, since $i< j < m < s,$ non comparable factors of $Y_{i,j}Y_{m,s}$ are $p_{1,m}p_{2,j},$ $(\prod\limits_{l=j+1}^{m}p_{l,\frac{n}{2}+l-2})(\prod\limits_{l=j}^{m-1}p_{l,\frac{n}{2}+l}).$

By using straightening law we replace $p_{1,m}p_{2,j}$ by  $p_{1,j}p_{2,m}-p_{1,2}p_{j,m}.$
 Since we are working with $X_{w}^{v_{k}}$ where $2\le k\le \frac{n}{2}-2,$ we have $v_{k}=(1,k
 +1)\ge (1,3).$ Therefore, we have $p_{12}=0$ identically on $X_{w}^{v_{k}}.$ 
 
On the other hand, we note that $(\prod\limits_{l=j+1}^{m}p_{l,\frac{n}{2}+l-2})(\prod\limits_{l=j}^{m-1}p_{l,\frac{n}{2}+l})=\prod\limits_{l=j}^{m-1}p_{l+1,\frac{n}{2}+l-1}p_{l,\frac{n}{2}+l}.$

By using straightening law we replace $p_{l+1,\frac{n}{2}+l-1}p_{l,\frac{n}{2}+l}$ by $p_{l,\frac{n}{2}+l-1}p_{l+1,\frac{n}{2}+l}-p_{l,l+1}p_{\frac{n}{2}+l-1,\frac{n}{2}+l}$ for all $j\le l\le m-1.$

Since $k+1\le i< j < m < s,$ and $k\ge 2,$ we have $\frac{n}{2}+l-1\ge \frac{n}{2}+j-1\ge \frac{n}{2}+3$ for all $j\le l\le m-1.$ Thus, $p_{l,l+1}p_{\frac{n}{2}+l-1,\frac{n}{2}+l}=0$ identically on $X_{w}^{v_{k}}.$

Therefore, after replacing $p_{1,m}p_{2,j}$  (respectively, $p_{l+1,\frac{n}{2}+l-1}p_{l,\frac{n}{2}+l}$ for all $j\le l\le m-1$) in $Y_{i,j}Y_{m,s}$ by $p_{1,j}p_{2,m}$ (respectively, $p_{l,\frac{n}{2}+l-1}p_{l+1,\frac{n}{2}+l}$ for all $j\le l\le m-1$) we get $Y_{i,j}Y_{m,s}=Y_{i,m}Y_{j,s}.$

Proof of $(ii)$: 
	
Since $i < j <m <s,$ $Y_{i,m}$ and $Y_{j,s}$ can be written as

$Y_{i,m}=p_{1,i}p_{2,m}(\prod\limits_{l=3}^{i-1}p_{l,\frac{n}{2}+l})(\prod\limits_{l=i+1}^{j-1}p_{l,\frac{n}{2}+l-1})(\prod\limits_{l=j}^{m-1}p_{l,\frac{n}{2}+l-1})(\prod\limits_{l=m+1}^{s-1}p_{l,\frac{n}{2}+l-2})(\prod\limits_{l=s}^{\frac{n}{2}+2}p_{l,\frac{n}{2}+l-2}).$ 

$ Y_{j,s}=p_{1,j}p_{2,s}(
\prod\limits_{l=3}^{i-1}p_{l,\frac{n}{2}+l})(\prod\limits_{l=i}^{j-1}p_{l,\frac{n}{2}+l})(\prod\limits_{l=j+1}^{m-1}p_{l,\frac{n}{2}+l-1})(\prod\limits_{l=m}^{s-1}p_{l,\frac{n}{2}+l-1})(\prod\limits_{l=s+1}^{\frac{n}{2}+2}p_{l,\frac{n}{2}+l-2}).$

Then we observe that 

$p_{1,i}p_{2,s}(\prod\limits_{l=3}^{i-1}p_{l,\frac{n}{2}+l})(\prod\limits_{l=i+1}^{m-1}p_{l,\frac{n}{2}+l-1})(\prod\limits_{l=m}^{s-1}p_{l,\frac{n}{2}+l-1})(\prod\limits_{l=s+1}^{\frac{n}{2}+2}p_{l,\frac{n}{2}+l-2})$ and

$p_{1,j}p_{2,m}
(\prod\limits_{l=3}^{j-1}p_{l,\frac{n}{2}+l})(\prod\limits_{l=j+1}^{m-1}p_{l,\frac{n}{2}+l-1})(\prod\limits_{l=m+1}^{\frac{n}{2}+2}p_{l,\frac{n}{2}+l-2})$ 

are factors of $Y_{i,m}Y_{j,s}.$ Therefore, we have $Y_{i,m}Y_{j,s}=Y_{i,s}Y_{j,m}.$
\end{proof}

\begin{lemma}\label{lemma3.10}
The ring $R_{v_{k},w}$ is generated in degree $1.$
\end{lemma}
\begin{proof}
Proof follows from \cref{cor3.3}.
\end{proof}
Let $\tilde{R}= \mathbb{C}[y_{i,j}]$ be the polynomial ring generated by $y_{i,j},$ where $k+1 \leq i < j \leq \frac{n}{2}+2.$ Then by \cref{lemma3.10} we have a surjective ring homomorphism $\phi: \tilde{R} \longrightarrow R_{v_{k},w}$ by sending $y_{i,j}$ to $Y_{i,j}.$ Let $I$ be the ideal generated by the following relations in $\tilde{R}.$
\begin{itemize}
	\item[(i)] $y_{i,j}y_{m,s}=y_{i,m}y_{j,s},$ 
\item[(ii)] $y_{i,j}y_{m,s}=y_{i,s}y_{j,m},$
\end{itemize}
for all  $k+1 \leq i < j < m < s \leq \frac{n}{2}+2.$ 
\begin{theorem}\label{Th3.11}
The map $\phi$ induces an isomorphism to $\tilde{\phi}:\tilde{R}/I \longrightarrow R_{v_{k},w}.$
\end{theorem}
\begin{proof}
We show that the diamond lemma of ring theory holds for this reduction system which implies is that any monomial in the $y_{i,j}$'s reduces, after applying these reductions (in any order, when multiple reduction rules apply) to a unique expression in the $y_{i,j}$'s, in which no term is divisible by a term appearing in the left hand side of the above reduction system (see \cite{Ber}).
	
We prove that the diamond lemma holds for this reduction system by looking at $y_{i,j}y_{m,s}y_{p,q},$ $1 \leq i < j < m < s < p < q \leq \frac{n}{2}+2.$

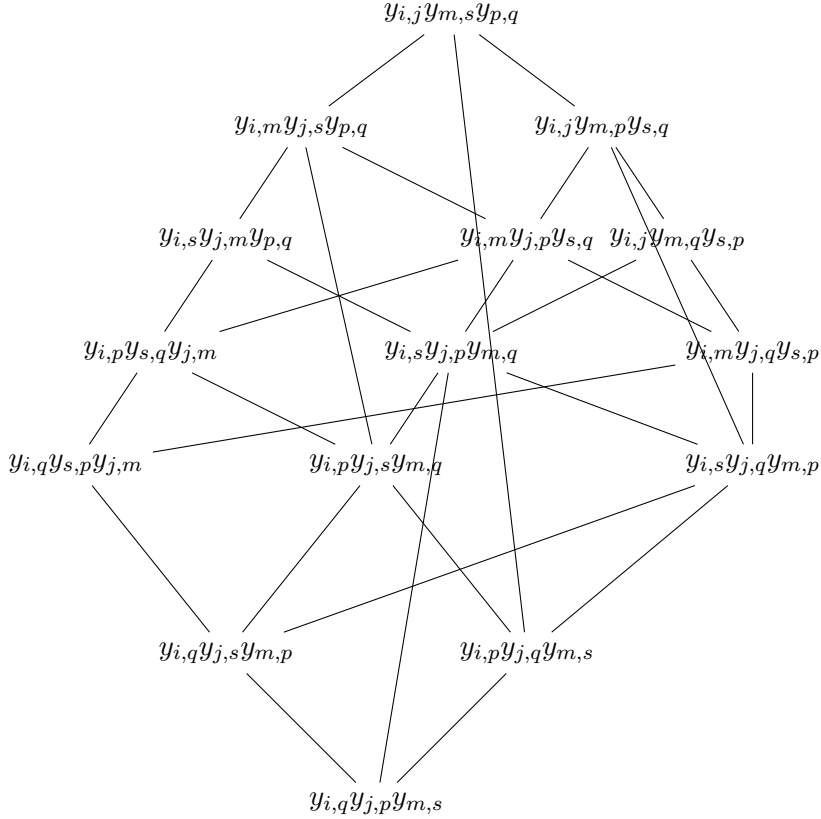
\begin{tikzpicture}[scale=.5]
\node (a) at (0,0)  {$y_{i,j}y_{m,s}y_{p,q}$};
\node (b) at (-4,-3) {$y_{i,m}y_{j,s}y_{p,q}$};
\node (d) at (4,-3) {$y_{i,j}y_{m,p}y_{s,q}$};
\node (e) at (-6,-6) {$y_{i,s}y_{j,m}y_{p,q}$};
\node (g) at (2,-6) {$y_{i,m}y_{j,p}y_{s,q}$};
\node (h) at (-8,-9) {$y_{i,p}y_{s,q}y_{j,m}$};
\node (i) at (0,-9) {$y_{i,s}y_{j,p}y_{m,q}$};
\node (k) at (-10,-12) {$y_{i,q}y_{s,p}y_{j,m}$};
\node (j) at (-2,-12) {$y_{i,p}y_{j,s}y_{m,q}$};
\node (l) at (-6,-17) {$y_{i,q}y_{j,s}y_{m,p}$};
\node (m) at (-2,-21) {$y_{i,q}y_{j,p}y_{m,s}$};
\node (n) at (2,-17) {$y_{i,p}y_{j,q}y_{m,s}$};
\node (p) at (8,-12) {$y_{i,s}y_{j,q}y_{m,p}$};
\node (f) at (8,-9) {$y_{i,m}y_{j,q}y_{s,p}$};
\node (q) at (6,-6) {$y_{i,j}y_{m,q}y_{s,p}$};
\draw (a) -- (b)--(e);
\draw (b)--(j);
\draw (b)--(g);
\draw (a) -- (n);
\draw (a) -- (d);
\draw (e) --(h);
\draw (e) -- (i);
\draw (h) -- (k) -- (l) --(m);
\draw (h) -- (j);
\draw (j) -- (n) -- (m);
\draw (p) -- (i) -- (j);
\draw (i) -- (m);
\draw (p) -- (l);
\draw (p) -- (n);
\draw (j) -- (l);
\draw (g) -- (h);
\draw (g) -- (i);
\draw (g) -- (f);
\draw (f) -- (k);
\draw (f) -- (p);
\draw (d) -- (g);
\draw (d) -- (p);
\draw (d) -- (q);
\draw (q) -- (f);
\draw (q) -- (i);
\end{tikzpicture}

Note that in the above diagram an edge between two vertices means one is obtained from other by applying a reduction. We see that after applying the reductions the monomial $y_{i,j}y_{m,s}y_{p,q}$ reduces to the unique expression $y_{i,q}y_{j,p}y_{m,s}.$ Therefore, $\tilde{\phi}$ is an isomorphism.	
\end{proof}
As a consequence of Theorem \ref{Th3.11} we have the following
\begin{corollary}\label{Cor 3.12}
The GIT quotient of $T\backslash \backslash (X^{v_{k}}_w)^{ss}_T(\mathcal{L}(\frac{n}{2}\omega_2)$ is a toric variety for $2\le k\le \frac{n}{2}-2$. 
\end{corollary}
\begin{proof} 
It follows from Theorem \ref{Th3.11} that the ideal $I$ of $R_{v_{k},w}$ is a prime ideal and generated by differences of two monomials. Thus, by \cite[Proposition 1.1.11.,p.16]{Cox}, the affine variety $Spec(R_{v_{k},w})$ is a toric variety. Therefore, by \cite[Proposition 2.1.4.,p.56]{Cox},  $T\backslash \backslash (X^{v_{k}}_w)^{ss}_T(\mathcal{L}(\frac{n}{2}\omega_2)=Proj(R_{v_{k},w})$ is a toric variety.
\end{proof}

\section{Torus quotient of $G_{2,6}$}\label{section4}
We recall that the Grassmannian $G_{2,n}$ is a closed subvariety of $\mathbb{P}^{N-1}$ where $N=\begin{pmatrix}
n\\
2
\end{pmatrix},$ given by the Pl\"ucker embedding. In this section, first we recall the Pl\"ucker relations for $G_{2,n}$ which will be used later.
\begin{theorem}
The Grassmannian $G_{2,n}$ consists of the zeroes in $\mathbb{P}^{N-1}$ of the following quadratic polynomials:
\begin{center}
$p_{i,l}p_{j,k}-p_{i,k}p_{j,l}+p_{i,j}p_{k,l}$
\end{center}
where $1\le i<j<k<l\le n$
(A relation obtained by equating an expression to $0$ is usually referred to as a Pl\"ucker relation).
\end{theorem}
\begin{proof}
See \cite[Theorem 4.1.3.1,p.31]{LS}.
\end{proof}

In this section, we study the GIT quotient of $G_{2,6}.$ 

Let $X= T\backslash\backslash (G_{2,6})^{ss}_T(\mathcal{L}(3\omega_2)).$
Let $R=\bigoplus\limits_{d\ge 0} R_{d},$ where  $R_{d}=H^0(G_{2,6}, \mathcal{L}^{\otimes d}(3\omega_{2}))^{T}.$ We note that $X=Proj(R)$ and $R_{d}$'s   are finite dimensional vector space. Let  $X_{1}=p_{14}p_{25}p_{36},$ $X_{2} =p_{12}p_{35}p_{46},$ $X_{3}=p_{13}p_{25}p_{46},$ $X_{4}=p_{12}p_{34}p_{56}$ and $X_{5}=p_{13}p_{24}p_{56}.$ Then the set $\{X_{1}, X_{2}, X_{3}, X_{4}, X_{5}\}$ forms a basis of $R_{1}.$  Let $Y_{1}=p_{12}p_{14}p_{24}p_{35}p_{36}p_{56}$ $Y_{2}=p_{12}p_{13}p_{23}p_{45}p_{46}p_{56},$ $W_{1}=X_4Y_2=p_{12}^{2}p_{13}p_{23}p_{34}p_{45}p_{46}p_{56}^{2},$ $W_{2}=X_2Y_2=p_{12}^{2}p_{13}p_{23}p_{35}p_{45}p_{46}^{2}p_{56},$ $W_{3}=X_5Y_2=p_{12}p_{13}^{2}p_{23}p_{24}p_{45}p_{46}p_{56}^{2}.$

In the following lemma we find out the relation among $X_i$'s and $Y_i$'s in $R_2$. 

\begin{lemma}\label{lemma4.1}
	Let $X_i$'s and $Y_i$'s be as above. Then we have the followings:
	
	\begin{itemize}
		\item[(1)]
		$X_1X_4=Y_1- X_2X_5+ X_4X_5- X_4^{2} +X_2X_4$. \hspace{2cm} 
		
		\item[(2)]
		$X_3X_4= X_2X_5 -Y_2$. \hspace{6cm} 
	\end{itemize}
\end{lemma}

\begin{proof}
	Proof of $(1):$ Note that 	
	\begin{equation*}
	\begin{aligned}
	X_1X_4 & =p_{12}p_{14}p_{25}p_{34}p_{36}p_{56} [\text{by using}~ p_{25}p_{34}=p_{35}p_{24}-p_{45}p_{23}]\\
	       & = p_{12}p_{14}p_{36}p_{56}(p_{35}p_{24}-p_{45}p_{23})\\
	        & =p_{12}p_{14}p_{24}p_{35}p_{36}p_{56}-p_{12}(p_{14}p_{23})p_{36}p_{45}p_{56}\\ 
	        & =p_{12}p_{14}p_{24}p_{35}p_{36}p_{56}-p_{12}(p_{13}p_{24}-p_{12}p_{34})p_{36}p_{45}p_{56}, [\text{by using}~ p_{14}p_{23}=p_{13}p_{24}-p_{12}p_{34}]\\
	        & =p_{12}p_{14}p_{24}p_{35}p_{36}p_{56}-p_{12}p_{13}p_{24}(p_{36}p_{45})p_{56}+p_{12}^2p_{34}(p_{36}p_{45})p_{56}\\
	       &=p_{12}p_{14}p_{24}p_{35}p_{36}p_{56}-p_{12}p_{13}p_{24}(p_{35}p_{46}-p_{34}p_{56})p_{56}\\
	       & \hspace{.5cm}+p_{12}^2p_{34}(p_{35}p_{46}-p_{34}p_{56})p_{56} [\text{by using}~ p_{36}p_{45}=p_{35}p_{46}-p_{34}p_{56}]\\ 
	       &=p_{12}p_{14}p_{24}p_{35}p_{36}p_{56}-p_{12}p_{13}p_{24}p_{35}p_{46}p_{56}+p_{12}p_{13}p_{24}p_{34}p_{56}^2+p_{12}^2p_{34}p_{35}p_{46}p_{56}-p_{12}^2p_{34}^2p_{56}^2\\
	        &=Y_1-X_2X_5+X_4X_5+X_2X_4-X_4^2.
	\end{aligned}
	\end{equation*}
	
	Proof of $(2):$ Note that 
	\begin{equation*}
	\begin{aligned}
	X_3X_4 &=p_{12}p_{13}p_{25}p_{34}p_{46}p_{56}\\
	&=p_{12}p_{13}p_{24}p_{35}p_{46}p_{56}-p_{12}p_{13}p_{23}p_{45}p_{46}p_{56} [\text{by using}~ p_{25}p_{34}=p_{24}p_{35}-p_{23}p_{45}]\\
	&=X_2X_5-Y_2.
	\end{aligned}
	\end{equation*} 
\end{proof}

\begin{lemma}\label{lemma4.2}
Let $X_i$'s and $W_i$'s be as above. Then we have the followings:
	
	\begin{itemize}

		\item[(1)] $X_4^2X_3=X_2X_4X_5-W_1$

		\item[(2)] $X_1X_3X_4=X_1X_2X_5-X_2X_3X_5+X_2^2X_5-W_2+X_2X_5^2-W_3-X_2X_4X_5+W_1$

		\item[(3)] $X_2X_3X_4=X_2^2X_5-W_2$
		
		\item[(4)] $X_3X_4X_5=X_2X_5^2-W_3$ 
		
	\end{itemize}
\end{lemma}
\begin{proof}

Proof of $(1)$: By using \cref{lemma4.1}$(2)$, 
$X_3X_4^2=X_2X_4X_5-W_1$.

Proof of $(2)$: By \cref{lemma4.1}$(2)$, 
$X_1X_3X_4=X_1X_2X_5-X_1Y_2$.

On the other hand, using the Pl\"{u}cker relations 
\begin{equation*}
\begin{aligned} 
X_1Y_2 & =p_{12}p_{13}(p_{14}p_{23})p_{25}(p_{36}p_{45})p_{46}p_{56}\\
       &=p_{12}p_{13}(p_{13}p_{24}-p_{12}p_{34})p_{25}(p_{35}p_{46}-p_{34}p_{56})p_{46}p_{56} ~[\text{by using } p_{14}p_{23}=p_{13}p_{24}-p_{12}p_{34},\\
       & \hspace{.5cm} \text{ and } p_{36}p_{45}=p_{35}p_{46}-p_{34}p_{56}]\\
       &=p_{12}p_{13}^2p_{24}p_{25}p_{35}p_{46}^2p_{56}-p_{12}^2p_{13}(p_{34}p_{25})p_{35}p_{46}^2p_{56}\\
      &\hspace{.5cm}-p_{12}p_{13}^2p_{24}(p_{25}p_{34})p_{46}p_{56}^2+p_{12}^2p_{13}p_{34}(p_{34}p_{25})p_{46}p_{56}^2\\
      &=p_{12}p_{13}^2p_{24}p_{25}p_{35}p_{46}^2p_{56}-p_{12}^2p_{13}(p_{24}p_{35}-p_{23}p_{45})p_{35}p_{46}^2p_{56}\\
      &\hspace{.5cm} -p_{12}p_{13}^2p_{24}(p_{24}p_{35}-p_{23}p_{45})p_{46}p_{56}^2+p_{12}^2p_{13}p_{34}(p_{24}p_{35}-p_{23}p_{45})p_{46}p_{56}^2 \\
      & \hspace{.5cm}[\text{by using } p_{25}p_{34}=p_{24}p_{35}-p_{23}p_{45}]\\
      & = p_{12}p_{13}^2p_{24}p_{25}p_{35}p_{46}^2p_{56}-p_{12}^2p_{13}p_{24}p_{35}^2p_{46}^2p_{56}+p_{12}^2p_{13}p_{23}p_{35}p_{45}p_{46}^2p_{56}\\
      & \hspace{.5cm} -p_{12}p_{13}^2p_{24}^2p_{35}p_{46}p_{56}^2+ p_{12}p_{13}^2p_{23}p_{24}p_{45}p_{46}p_{56}^2+p_{12}^2p_{13}p_{24}p_{34}p_{35}p_{46}p_{56}^2\\
      &\hspace{.5cm} -p_{12}^2p_{13}p_{23}p_{34}p_{45}p_{46}p_{56}^2\\
      &=X_2X_3X_5-X_2^2X_5+W_2-X_5^2X_2+W_3+X_2X_4X_5-W_1.
\end{aligned}
\end{equation*}

Therefore, $X_1X_3X_4=X_1X_2X_5-X_2X_3X_5+X_2^2X_5-W_2+X_5^2X_2-W_3-X_2X_4X_5+W_1.$

Proof of $(4)$: By using \cref{lemma4.1}$(2)$,
$X_2X_3X_4=X_2^2X_5-W_2.$

Proof of $(5)$: By using \cref{lemma4.1}$(2)$,
$X_3X_4X_5=X_2X_5^2-W_3.$
\end{proof}
Recall that by \cref{cor3.3}, $R$ is a graded $\mathbb{C}$-algebra which is an integral domain and generated by $R_1$ elements $X_1,X_2,X_3,X_4,X_5.$

Let $\tilde{R}=\mathbb{C}[x_{1}, x_{2}, x_{3}, x_{4}, x_{5}]$ be the polynomial ring in $x_{i}$'s for $1\le i\le 5.$
\begin{lemma}\label{lemma4.3}
There is no linear or quadratic relation among $X_{i}$'s in $R.$
\end{lemma}
\begin{proof}
Since $\{X_i: 1\le i\le 5\}$ are standard monomials of degree one, there is no linear relation among $X_i$'s in $R.$

We next prove that there is no quadratic relation among $X_i$'s in R. Suppose there is a quadratic relation $\sum\limits_{i=1}^{5}c_iX_i^{2} +\sum\limits_{1\le i<j\le 5}c_{i,j}X_{i}X_{j}=0.$ 

$\sum\limits_{i=1}^{5}c_iX_i^{2} +
\sum\limits_{\substack{1\le i<j\le 5\\(i,j)\neq (1,4),(3,4)}}c_{i,j}X_{i}X_{j}+ c_{1,4}X_1X_4 +c_{3,4}X_3X_4=0.$

By using \cref{lemma4.1} we get   
\begin{center}
$\sum\limits_{i=1}^{5}c_iX_i^{2}+\sum\limits_{\substack{1\le i<j\le 5\\(i,j)\neq (1,4),(3,4)}}c_{i,j}X_{i}X_{j}+ c_{1,4}(Y_1- X_2X_5+ X_4X_5- X_4^{2} +X_2X_4)+ c_{3,4}(X_2X_5-Y_2)=0$ 
\end{center}
Since standard monomials of degree 2 are linearly independent, the coefficients  of $Y_1$ and $Y_2$ are zero i.e., $c_{1,4}=0,c_{3,4}=0.$

Thus from the above 
\begin{center}
$\sum\limits_{i=1}^{5}c_iX_i^{2}+\sum\limits_{\substack{1\le i<j\le 5\\(i,j)\neq (1,4),(3,4)}}c_{i,j}X_{i}X_{j}=0.$ 
\end{center}
Since $X_iX_j$'s are standard monomials of degree $2$ for $(i,j)\neq (1,4),(3,4),$  $c_{i}=0$ for $1\le i< j\le 5$ and $c_{i,j}=0$ for all $1\le i<j\le 5$ such that $(i,j)\neq (1,4),(3,4).$ Therefore, there is no quadratic relation among $X_i$'s in $R.$

\end{proof}

\begin{lemma}\label{lemma4.4}
$F(X_{1},X_2,X_3,X_4,X_5)=0$ in $R,$ where   $F(x_{1},x_2,x_3,x_4,x_5)=x_3x_4^{2}-x_{1}x_{2}x_{5}+x_1x_3x_4-x_2x_3x_4+ x_2x_3x_5-x_3x_4x_5 \in \tilde{R}.$
\end{lemma}
\begin{proof}
Recall that by \cref{lemma4.2} we have the following:

$X_1X_3X_4=X_1X_2X_5-X_2X_3X_5+X_2^2X_5-W_2+X_2X_5^2-W_3-X_2X_4X_5+W_1.$ $\hspace{1cm} (2.8.1)$

$X_4^2X_3=X_2X_4X_5-W_1.$

$X_2X_3X_4=X_2^2X_5-W_2.$

$X_3X_4X_5=X_2X_5^2-W_3.$

Substituting  $W_1, W_2,$ and $W_3$ in $(2.8.1)$ we get

$X_1X_3X_4=X_1X_2X_5-X_2X_3X_5+X_2^2X_5+X_2X_3X_4-X_2^2X_5+X_2X_5^2+X_3X_4X_5-X_2X_5^2-X_2X_4X_5-X_4^2X_3+X_2X_4X_5.$

Therefore, $X_3X_4^2-X_1X_2X_5+X_1X_3X_4-X_2X_3X_4+X_2X_3X_5-X_3X_4X_5=0.$

\end{proof}

\begin{corollary}\label{cor4.5}
$F(x_{1},x_2,x_3,x_4,x_5)=x_3x_4^{2}-x_{1}x_{2}x_{5}+x_1x_3x_4-x_2x_3x_4+ x_2x_3x_5-x_3x_4x_5$ is an irreducible polynomial in $\tilde{R}.$
\end{corollary}
\begin{proof}
	
Suppose $F(x_{1},x_2,x_3,x_4,x_5)$ is not an irreducible element of $\tilde{R}.$ Then there exist two non unit elements
$F_1(x_{1},x_2,x_3,x_4,x_5), F_2(x_{1},x_2,x_3,x_4,x_5)$ in $\tilde{R}$ such that $F(x_{1},x_2,x_3,x_4,x_5)=F_1(x_{1},x_2,x_3,x_4,x_5)F_2(x_{1},x_2,x_3,x_4,x_5).$ 
	
By Lemma \ref{lemma4.4}, $F(X_{1},X_2,X_3,X_4,X_5)=0.$ Since $R$ is an integral domain and \linebreak$F_1(X_{1},X_2,X_3,X_4,X_5)F_2(X_{1},X_2,X_3,X_4,X_5)=0$, there is either a linear relation or a quadratic relation among $X_{i}$'s in $R.$ This contradicts the Lemma \ref{lemma4.3}. 
\end{proof}
Let $I=(x_3x_4^{2}-x_{1}x_{2}x_{5}+x_1x_3x_4-x_2x_3x_4+ x_2x_3x_5-x_3x_4x_5).$ Note that since  $I$ is generated by irreducible element, and $\tilde{R}$ is UFD, $I$ is prime ideal in $\tilde{R}.$ Then $\tilde{R}/I$ is an integral domain. Let $\pi: \tilde{R}\longrightarrow \tilde{R}/I$ be the natural surjective homomorphism of $\mathbb{C}$-algebras. Let $\varphi: \tilde{R}\longrightarrow R$ be the $\mathbb{C}$-algebra homomorphism defined by $x_{i}\mapsto X_i.$ By Lemma \ref{lemma4.4}, $\varphi$ induces  $\mathbb{C}$-algebra homomorphism $\tilde{\varphi}:\tilde{R}/I\longrightarrow R.$ 

Note that by \cref{lemma3.1}, $(G/P_{\hat{\alpha_{2}}})^{s}_T(\mathcal{L}(3\omega_{2}))\neq \emptyset$. Therefore, Krull dimension of $R$ is $4.$ On the other hand, Krull dimension of $\tilde{R}/I$ is $4.$
\begin{lemma}\label{lemma4.6}
$\tilde{\varphi}:\tilde{R}/I\longrightarrow R$ is an isomorphism of $\mathbb{C}$-algebras.
\end{lemma}
\begin{proof}
It is clear that $\tilde{\varphi}$ is a surjective morphism of $\mathbb{C}$-algebras. Since $R$ is an integral domain, $Ker(\tilde{\varphi})$ is a prime ideal of $\tilde{R}/I.$ 
Therefore, by \cite[Theorem 1.8A.(b), p.6]{R}, we have $dim((\tilde{R}/I)/Ker(\tilde{\varphi}))$ + $ht(Ker(\tilde{\varphi}))=dim(\tilde{R}/I).$ On the other hand, we have  $dim((\tilde{R}/I)/Ker(\tilde{\varphi}))=dim(R)=dim(\tilde{R}/I).$ Hence, $ht(Ker(\tilde{\varphi}))=0,$ i.e. $Ker(\tilde{\varphi})=0.$ Therefore, $\tilde{\varphi}$ is an isomorphism of $\mathbb{C}$-algebras. 
\end{proof}
\begin{corollary}
$\tilde{\varphi}:\tilde{R}/I\longrightarrow R$ is an isomorphism as graded $\mathbb{C}$-algebras.
\end{corollary}
\begin{corollary}
$X$ is a complete intersection variety. 
\end{corollary}

\section{Torus quotient of $X(6,8)$ in $G_{2,8}$}\label{section5}

In this section, our aim is to study the GIT quotient of $X(6,8)$ in $G_{2,8}.$

We consider the Schubert variety $X(6,8)$ in the Grassmannian $G_{2,8}$. Let $R=\bigoplus\limits_{d\ge 0} R_{d},$ where  $R_{d}=H^0(X(6,8), \mathcal{L}^{\otimes d}(4\omega_{2}))^{T}.$ We note that $X=Proj(R)$ and $R_{d}$'s   are finite dimensional vector space. Let  $X_{1}=p_{12}p_{34}p_{57}p_{68},$ $X_{2} =p_{12}p_{35}p_{47}p_{68},$ $X_{3}=p_{13}p_{25}p_{47}p_{68},$ $X_{4}=p_{13}p_{24}p_{57}p_{68}$, $X_{5}=p_{14}p_{25}p_{37}p_{68}$, $X_{6}=p_{14}p_{26}p_{37}p_{58}$, $X_{7}=p_{13}p_{26}p_{47}p_{58}$, $X_{8}=p_{12}p_{36}p_{47}p_{58}$ and $X_{9}=p_{15}p_{26}p_{37}p_{48}$. Then the set $\{X_{1}, X_{2}, X_{3}, X_{4}, X_{5}, X_{6}, X_{7}, X_{8}, X_{9}\}$ forms a basis of $R_{1}.$ Let $Y_{1}=p_{12}p_{13}p_{23}p_{45}p_{47}p_{57}p^2_{68}$, $Y_{2}=p_{12}p_{14}p_{24}p_{35}p_{37}p_{57}p^2_{68}$, $Y_{3}=p_{12}p_{14}p_{24}p_{36}p_{37}p_{57}p_{58}p_{68}$, $Y_{4}=p_{12}p_{13}p_{23}p_{46}p_{47}p_{57}p_{58}p_{68}$, 
$Y_{5}=p_{12}p_{15}p_{25}p_{36}p_{37}p_{47}p_{48}p_{68}$.

\begin{lemma}\label{lemma5.1}
We have the following relations in $R_{2}.$
\begin{itemize}
\item [(1)] $X_{1}X_{3}=X_{2}X_{4}-Y_{1}.$

\item [(2)] $X_{1}X_{5}=Y_{2}-X_{2}X_{4}+X_{1}X_{2}+X_{1}X_{4}-X_{1}^{2}.$

\item [(3)] $X_{1}X_{6}=Y_{3}-X_{4}X_{8}+X_{1}X_{8}+X_{1}X_{4}-X_{1}^{2}.$

\item [(4)] $X_{1}X_{7}=X_{4}X_{8}-Y_{4}.$

\item [(5)] $X_{1}X_{9}=X_{5}X_{8}-X_{3}X_{8}+X_{1}X_{8}+X_{2}X_{4}-X_{1}X_{2}-Y_{1}.$

\item [(6)] $X_{2}X_{6}=X_{5}X_{8}-X_{4}X_{8}+X_{1}X_{8}+X_{2}X_{4}-X_{1}X_{2}.$

\item [(7)] $X_{2}X_{7}=X_{3}X_{8}-Y_{4}+Y_{1}.$

\item [(8)] $X_{2}X_{9}=Y_{5}-X_{3}X_{8}+X_{2}X_{8}+X_{2}X_{3}-X_{2}^{2}.$

\item [(9)] $X_{4}X_{9}=X_{3}X_{6}-X_{3}X_{8}+X_{4}X_{8}-Y_{1}.$
\end{itemize}
\end{lemma}
\begin{proof}
Proof of (1): \begin{equation*}
\begin{aligned} 
X_1X_3 & =p_{12}p_{13}(p_{25}p_{34})p_{47}p_{57}p_{68}^{2}\\
&=p_{12}p_{13}(p_{24}p_{35}-p_{23}p_{45})p_{47}p_{57}p_{68}^{2}\\
&=p_{12}p_{13}p_{24}p_{35}p_{47}p_{57}p_{68}^{2}-p_{12}p_{13}p_{23}p_{45}p_{47}p_{57}p_{68}^{2}\\
&=X_{2}X_{4}-Y_{1}.
\end{aligned}
\end{equation*}
Proof of (2),(3),(4),(5),(6),(7),(8),(9) are similar to (1).
\end{proof}

\begin{lemma}\label{lemma5.2}
We have the following relations in $R_2.$ 
\begin{itemize}
		
\item[(1)] $X_2X_6-X_5X_8+X_4X_8-X_1X_8-X_2X_4+X_1X_2=0$
		
\item[(2)] $X_1X_3-X_2X_4+X_3X_6-X_3X_8+X_4X_8-X_4X_9=0$
		
\item[(3)] $X_2X_7-X_1X_7-X_3X_6+X_4X_9=0$
		
\item[(4)] $X_3X_6-X_5X_7=0$
		
\item[(5)] $X_1X_3-X_2X_4+X_2X_6-X_3X_8+X_4X_8-X_1X_9=0$.
		
\end{itemize}
\end{lemma}
\begin{proof}
Proof of (1): Follows from Lemma \ref{lemma5.1}(6).

Proof of (2): Follows from Lemma \ref{lemma5.1}(1) and Lemma \ref{lemma5.1}(9) eliminating $Y_{1}.$

Proof of (3): Follows from Lemma \ref{lemma5.1}(4), Lemma \ref{lemma5.1}(7), and Lemma \ref{lemma5.1}(9) eliminating $Y_{1}, Y_{4}.$

Proof of (4): $X_{3}X_{6}-X_{5}X_{7}=p_{13}p_{25}p_{47}p_{68}p_{14}p_{26}p_{37}p_{58}-p_{14}p_{25}p_{37}p_{68}p_{13}p_{26}p_{47}p_{58}=0.$

Proof of (5): By Lemma \ref{lemma5.1}(1), Lemma \ref{lemma5.1}(5) eliminating $Y_{1},$ we have $X_{1}X_{3}-X_{2}X_{4}-X_{1}X_{9}+X_{5}X_{8}-X_{3}X_{8}+X_{1}X_{8}+X_{2}X_{4}-X_{1}X_{2}=0.$ On the other hand,  by \cref{lemma5.2}(1), we have $X_{5}X_{8}+X_{1}X_{8}+X_{2}X_{4}-X_{1}X_{2}=X_{2}X_{6}+X_{4}X_{8}.$ Therefore,  $X_{1}X_{3}-X_{2}X_{4}+X_{2}X_{6}-X_{3}X_{8}+X_{4}X_{8}-X_{1}X_{9}=0.$ 
\end{proof}

We now use Macaulay 2 to see that the ideal $I$ generated by the above relations is a prime ideal in the ring $R$ generated by $X_1, X_2, X_3, X_4, X_5, X_6, X_7, X_8, X_9$. 

\scriptsize{
\begin{lstlisting}
i1 : R=QQ[x_1,x_2,x_3,x_4,x_5,x_6,x_7,x_8,x_9]

o1 = R

o1 : PolynomialRing

i2 :  I=ideal{x_2*x_6-x_5*x_8+x_4*x_8-x_1*x_8-x_2*x_4+x_1*x_2,x_1*x_3
-x_2*x_4+x_3*x_6-x_3*x_8+x_4*x_8-x_4*x_9,x_2*x_7-x_1*x_7-x_3*x_6+x_4*
x_9,x_3*x_6-x_5*x_7,x_1*x_3-x_2*x_4+x_2*x_6-x_3*x_8+x_4*x_8-x_1*x_9}

o2 = ideal (x x  - x x  + x x  - x x  + x x  - x x , x x  - x x    
             1 2    2 4    2 6    1 8    4 8    5 8   1 3    2 4 
             
     ---------------------------------------------------------------
             
      + x x  - x x  + x x  - x x , -x x - x x  + x x  + x x , 
         3 6    3 8    4 8    4 9    3 6   1 7    2 7    4 9   
     ---------------------------------------------------------------
      x x  - x x , x x  - x x  + x x  - x x  + x x  - x x )
       3 6    5 7   1 3    2 4    2 6    3 8    4 8    1 9

o2 : Ideal of R

i3 : isPrime I

o3 = true
\end{lstlisting}
}

\normalsize{Now we see the dimension (Krull dimension) of the quotient ring $A=R/I$ is $5$.}

\scriptsize{
\begin{lstlisting}  

i4 : A=R/I

o4 = A

o4 : QuotientRing

i5 : dim A

o5 = 5
\end{lstlisting}
}

\normalsize{We note that by \cref{cor3.4}, $R$ is a graded $\mathbb{C}$-algebra which is an integral domain and generated by $R_1$ elements $X_i,$ $1\le i\le9.$} 

Let $A=\mathbb{C}[x_{1}, x_{2}, x_{3},..., x_{9}]$ be the polynomial ring in $x_{i}$'s for $1\le i\le 9.$

Let $I$ be the ideal of $A$ generated by the above relations. Then $I$ is a prime ideal of $A.$ Therefore, $A/I$ is an integral domain. Let $\pi: A\longrightarrow A/I$ be the natural surjective homomorphism of $\mathbb{C}$-algebras. Let $\varphi: A\longrightarrow R$ be the $\mathbb{C}$-algebra homomorphism defined by $x_{i}\mapsto X_i.$ By Lemma \ref{lemma5.2}, $\varphi$ induces  $\mathbb{C}$-algebra homomorphism $\tilde{\varphi}:A/I\longrightarrow R.$ 

Note that by \cref{lemma3.1}, $X(w)^{s}_{T}(\mathcal{L}(4\omega_{2}))\neq \emptyset$. Therefore, Krull dimension of $R$ is $5.$ On the other hand, Krull dimension of $A/I$ is $5.$
\begin{lemma}
$\tilde{\varphi}:A/I\longrightarrow R$ is an isomorphism of $\mathbb{C}$-algebras.
\end{lemma}
\begin{proof}
Proof is similar to the proof of Lemma \ref{lemma4.6}.  
\end{proof}
\begin{corollary}
$\tilde{\varphi}:A/I\longrightarrow R$ is an isomorphism as graded $\mathbb{C}$-algebras.
\end{corollary} 
\begin{corollary}
$X$ is not a complete intersection variety.
\end{corollary}

\section{Torus quotient of $X(7,10)$ in $G_{2,10}$}\label{section6}

In this section, we study the GIT quotient of $X(7,10)$ in $G_{2,10}.$
 
Let $w=(s_6s_5s_4s_3s_2s_1)(s_9s_8s_7s_6s_5s_4s_3s_2).$ In one line notation $w=(7,10).$
We consider the Schubert variety $X(w)$ in the Grassmannian $G_{2,10}.$ Let $R=\bigoplus\limits_{d\ge 0} R_{d},$ where  $R_{d}=H^0(X(w), \mathcal{L}^{\otimes d}(5\omega_{2}))^{T}.$ We note that $X=Proj(R)$ and $R_{d}$'s   are finite dimensional vector space. Let  $X_{1}=p_{12}p_{34}p_{58}p_{69}p_{710},$ $X_{2} =p_{12}p_{35}p_{48}p_{69}p_{710},$ $X_{3}=p_{12}p_{36}p_{48}p_{59}p_{710},$ $X_{4}=p_{12}p_{37}p_{48}p_{59}p_{610},$ $X_{5}=p_{13}p_{26}p_{48}p_{59}p_{710},$ $X_{6}=p_{13}p_{25}p_{48}p_{69}p_{710},$ $X_{7}=p_{13}p_{24}p_{58}p_{69}p_{710},$ $X_{8}=p_{13}p_{27}p_{48}p_{59}p_{610},$ $X_{9}=p_{14}p_{26}p_{38}p_{59}p_{710},$
$X_{10}=p_{14}p_{25}p_{38}p_{69}p_{710},$
$X_{11}=p_{14}p_{27}p_{38}p_{59}p_{610},$
$X_{12}=p_{15}p_{26}p_{38}p_{49}p_{710},$
$X_{13}=p_{15}p_{27}p_{38}p_{49}p_{610},$ and 
$X_{14}=p_{16}p_{27}p_{38}p_{49}p_{510}.$ Let $Y_1=p_{12}p_{13}p_{23}p_{47}p_{48}p_{58}p_{59}p_{69}p_{610}p_{710}$,  $Y_2=p_{12}p_{13}p_{23}p_{46}p_{48}p_{58}p_{59}\\p_{69}p_{710}^2$, $Y_3=p_{12}p_{13}p_{23}p_{45}p_{48}p_{58}p_{69}^2p_{710}^2$.

\begin{lemma}\label{lemma6.1}
	Let $X_i$'s and $Y_i$'s are as above. Then we have the following relations:
	
	\begin{itemize}
		\item [(1)] $Y_1=X_4X_7-X_1X_8$
		
		\item [(2)] $Y_2=X_3X_7-X_1X_5$
		
		\item [(3)] $Y_3=X_2X_7-X_1X_6.$
	\end{itemize}
\end{lemma}

\begin{proof}
	Proof of $(1)$: 
	\begin{equation*}
	\begin{aligned}
	X_1X_8 &=p_{12}p_{13}(p_{27}p_{34})p_{48}p_{58}p_{59}p_{69}p_{610}p_{710}\\
	&=p_{12}p_{13}p_{24}p_{37}p_{48}p_{58}p_{59}p_{69}p_{610}p_{710}-p_{12}p_{13}p_{23}p_{47}p_{48}p_{58}p_{59}p_{69}p_{610}p_{710}\\
	& \hspace{.5cm} [\text{by using}~ p_{27}p_{34}=p_{24}p_{37}-p_{23}p_{47}]\\
	&=X_4X_7-Y_1.
	\end{aligned}
	\end{equation*}
	
	Proof of $(2)$: 
	\begin{equation*}
	\begin{aligned}
	X_1X_5 &=p_{12}p_{13}(p_{26}p_{34})p_{48}p_{58}p_{59}p_{69}p_{710}^2\\
	&=p_{12}p_{13}p_{24}p_{36}p_{48}p_{58}p_{59}p_{69}p_{710}^2-p_{12}p_{13}p_{23}p_{46}p_{48}p_{58}p_{59}p_{69}p_{710}^2\\
	& \hspace{.5cm} [\text{by using}~ p_{26}p_{34}=p_{24}p_{36}-p_{23}p_{46}]\\
	&=X_3X_7-Y_2.
	\end{aligned}
	\end{equation*}
	
	Proof of $(3)$: 
	\begin{equation*}
	\begin{aligned}
	X_1X_6 &=p_{12}p_{13}(p_{25}p_{34})p_{48}p_{58}p_{69}^2p_{710}^2\\
	&=p_{12}p_{13}p_{24}p_{35}p_{48}p_{58}p_{69}^2p_{710}^2-p_{12}p_{13}p_{23}p_{45}p_{48}p_{58}p_{69}^2p_{710}^2\\
	& \hspace{.5cm} [\text{by using}~ p_{26}p_{34}=p_{24}p_{36}-p_{23}p_{46}]\\
	&=X_2X_7-Y_3.
	\end{aligned}
	\end{equation*}
\end{proof}

\begin{lemma}\label{lemma6.2}
	In $R_2,$ the following relations among $X_i$'s hold :
	\begin{itemize}
		\item [(1)] $X_2X_9-X_3X_{10}+X_3X_7-X_1X_3-X_2X_7+X_1X_2=0$
		
		\item [(2)] $X_2X_{11}-X_4X_{10}+X_4X_7-X_1X_4-X_2X_7+X_1X_2=0$
		
		\item [(3)] $X_3X_8-X_4X_5-X_1X_8+X_4X_7-X_3X_7+X_1X_5=0$
		
		\item [(4)] $X_3X_{13}-X_4X_{12}+X_4X_6-X_2X_4-X_3X_6+X_2X_3=0$

		\item [(5)] $X_3X_{11}-X_4X_9+X_4X_7-X_1X_4-X_3X_7+X_1X_3=0$
		
		\item [(6)] $X_{10}X_{14}-X_9X_{13}+X_4X_9-X_4X_{10}+X_3X_7-X_1X_3-X_2X_7+X_1X_2=0$
		
		\item [(7)] $X_5X_{10}-X_6X_9=0$
		
		\item [(8)] $X_8X_{12}-X_5X_{13}=0$
		
		\item [(9)] $X_5X_{11}-X_8X_{9}=0$
		
		\item [(10)] $X_6X_{11}-X_8X_{10}=0$
		
		\item [(11)] $X_9X_{13}-X_{11}X_{12}=0$
		
		\item [(12)] $X_2X_8-X_4X_6-X_2X_7+X_4X_7-X_1X_8+X_1X_6=0$
		
		\item [(13)] $X_2X_5-X_3X_6+X_3X_7-X_1X_5+X_1X_6-X_2X_7=0$
		
		\item [(14)] $X_1X_{14}-X_4X_9+X_4X_5-X_1X_4+X_1X_3-X_1X_5=0$
		
		\item [(15)] $X_1X_{13}-X_4X_{10}+X_4X_6-X_1X_4-X_1X_6+X_1X_2=0$
		
		\item[(16)] $X_1X_{12}-X_3X_{10}+X_3X_6-X_1X_3-X_1X_6+X_1X_2=0$
		
		\item[(17)] $X_7X_{14}-X_8X_9+X_4X_5-X_4X_7+X_3X_7-X_1X_5=0$
		
		\item[(18)] $X_6X_{14}-X_8X_{12}+X_4X_5-X_4X_6+X_3X_7-X_1X_5+X_1X_6-X_2X_7=0$
		
		\item[(19)] $X_2X_{14}-X_4X_{12}+X_4X_5-X_2X_4-X_3X_6+X_2X_3+X_3X_7-X_1X_5+X_1X_6-X_2X_7=0$
		
		\item[(20)] $X_7X_{12}-X_5X_{10}+X_3X_6-X_3X_7+X_2X_7-X_1X_6=0$
		
		\item[(21)] $X_7X_{13}-X_8X_{10}+X_4X_6-X_4X_7+X_2X_7-X_1X_6=0.$
	\end{itemize}
\end{lemma}
\begin{proof}
	
	Proof of $(1)$:
	\begin{equation*}
	\begin{aligned}
	X_2X_9 & =p_{12}p_{14}(p_{26}p_{35})p_{38}p_{48}p_{59}p_{69}p_{710}^2\\
	& =p_{12}p_{14}(p_{25}p_{36}-p_{23}p_{56})p_{38}p_{48}p_{59}p_{69}p_{710}^2 ~[\text{by using}~ p_{26}p_{35}=p_{25}p_{36}-p_{23}p_{56}]\\
	& =p_{12}p_{14}p_{25}p_{36}p_{38}p_{48}p_{59}p_{69}p_{710}^2-p_{12}(p_{14}p_{23})p_{56}p_{38}p_{48}p_{59}p_{69}p_{710}^2\\
	& =p_{12}p_{14}p_{25}p_{36}p_{38}p_{48}p_{59}p_{69}p_{710}^2-p_{12}(p_{13}p_{24}-p_{12}p_{34})p_{56}p_{38}p_{48}p_{59}p_{69}p_{710}^2\\
	&\hspace{.5cm} [\text{by using}~ p_{14}p_{23}=p_{13}p_{24}-p_{12}p_{34}]\\
	& =p_{12}p_{14}p_{25}p_{36}p_{38}p_{48}p_{59}p_{69}p_{710}^2-p_{12}p_{13}p_{24}(p_{56}p_{38})p_{48}p_{59}p_{69}p_{710}^2+p_{12}^2p_{34}(p_{56}p_{38})p_{48}p_{59}p_{69}p_{710}^2\\
	& =p_{12}p_{14}p_{25}p_{36}p_{38}p_{48}p_{59}p_{69}p_{710}^2-p_{12}p_{13}p_{24}(p_{36}p_{58}-p_{35}p_{68})p_{48}p_{59}p_{69}p_{710}^2+p_{12}^2p_{34}(p_{36}p_{58}\\
	& \hspace{.5cm}-p_{35}p_{68})p_{48}p_{59}p_{69}p_{710}^2 ~[\text{by using}~ p_{38}p_{56}=p_{36}p_{58}-p_{35}p_{68}]\\
	\end{aligned}
\end{equation*}
\begin{equation*}
\begin{aligned}
	& =p_{12}p_{14}p_{25}p_{36}p_{38}p_{48}p_{59}p_{69}p_{710}^2-p_{12}p_{13}p_{24}p_{36}p_{48}p_{58}p_{59}p_{69}p_{710}^2+p_{12}p_{13}p_{24}p_{35}p_{48}(p_{68}p_{59})p_{69}p_{710}^2\\
	& \hspace{.5cm}+p_{12}^2p_{34}p_{36}p_{48}p_{58}p_{59}p_{69}p_{710}^2-p_{12}^2p_{34}p_{35}p_{48}(p_{68}p_{59})p_{69}p_{710}^2\\
	& =p_{12}p_{14}p_{25}p_{36}p_{38}p_{48}p_{59}p_{69}p_{710}^2-p_{12}p_{13}p_{24}p_{36}p_{48}p_{58}p_{59}p_{69}p_{710}^2+p_{12}p_{13}p_{24}p_{35}p_{48}(p_{58}p_{69}-\\
	& \hspace{.5cm}p_{56}p_{89})p_{69}p_{710}^2+p_{12}^2p_{34}p_{36}p_{48}p_{58}p_{59}p_{69}p_{710}^2-p_{12}^2p_{34}p_{35}p_{48}(p_{58}p_{69}-p_{56}p_{89})p_{69}p_{710}^2\\
	& \hspace{.5cm} [\text{by using}~ p_{59}p_{68}=p_{58}p_{69}-p_{56}p_{89}]\\
	& =p_{12}p_{14}p_{25}p_{36}p_{38}p_{48}p_{59}p_{69}p_{710}^2-p_{12}p_{13}p_{24}p_{36}p_{48}p_{58}p_{59}p_{69}p_{710}^2+p_{12}p_{13}p_{24}p_{35}p_{48}p_{58}p_{69}^2p_{710}^2\\
	& \hspace{.5cm}+p_{12}^2p_{34}p_{36}p_{48}p_{58}p_{59}p_{69}p_{710}^2-p_{12}^2p_{34}p_{35}p_{48}p_{58}p_{69}^2p_{710}^2 ~[\text{since } p_{89}|_{X(7,10)}=0]\\
	&=X_3X_{10}-X_3X_7+X_2X_7+X_1X_3-X_1X_2.
	\end{aligned}
	\end{equation*}

Proof of (2),(3),...,(21) are similar to (1).
\end{proof}

We now use Macaulay 2 to see that the ideal $I$ generated by the above relations is a prime ideal in the ring $R$ generated by $X_1, X_2, X_3, X_4, X_5, X_6, X_7, X_8, X_9, X_{10}, X_{11}, X_{12}, X_{13}, X_{14}$. 

\scriptsize{
	\begin{lstlisting}

i1 : R=QQ[x_1,x_2,x_3,x_4,x_5,x_6,x_7,x_8,x_9,x_10,x_11,x_12,x_13,x_14]
	
o1 = R
	
o1 : PolynomialRing
	
i2 : I=ideal{x_2*x_9-x_3*x_10+x_3*x_7-x_1*x_3-x_2*x_7+x_1*x_2,x_2*x_11-x_4*x_10+x
x_4*x_7-x_1*x_4-x_2*x_7+x_1*x_2,x_3*x_8-x_4*x_5-x_1*x_8+x_4*x_7-x_3*x_7+x_1*x_5,x
x_3*x_13-x_4*x_12+x_4*x_6-x_2*x_4-x_3*x_6+x_2*x_3,x_3*x_11-x_4*x_9+x_4*x_7-x_1*x_
_4-x_3*x_7+x_1*x_3,x_10*x_14-x_9*x_13+x_4*x_9-x_4*x_10+x_3*x_7-x_1*x_3-x_2*x_7+x_
_1*x_2,x_5*x_10-x_6*x_9,x_8*x_12-x_5*x_13,x_5*x_11-x_8*x_9,x_6*x_11-x_8*x_10,x_9*
*x_13-x_11*x_12,x_1*x_6-x_4*x_6-x_2*x_7+x_4*x_7-x_1*x_8+x_2*x_8,x_2*x_5-x_3*x_6+x
x_3*x_7-x_1*x_5+x_1*x_6-x_2*x_7,x_1*x_14-x_4*x_9+x_4*x_5-x_1*x_4+x_1*x_3-x_1*x_5,
,x_1*x_13-x_4*x_10+x_4*x_6-x_1*x_4-x_1*x_6+x_1*x_2,x_1*x_12-x_3*x_10+x_3*x_6-x_1*
*x_3-x_1*x_6+x_1*x_2,x_7*x_14-x_8*x_9+x_4*x_5-x_4*x_7+x_3*x_7-x_1*x_5,x_6*x_14-x_
_8*x_12+x_4*x_5-x_4*x_6+x_3*x_7-x_1*x_5+x_1*x_6-x_2*x_7,x_2*x_14-x_4*x_12+x_4*x_5
5-x_2*x_4-x_3*x_6+x_2*x_3+x_3*x_7-x_1*x_5+x_1*x_6-x_2*x_7,x_7*x_12-x_5*x_10+x_3*x
x_6-x_3*x_7+x_2*x_7-x_1*x_6,x_7*x_13-x_8*x_10+x_4*x_6-x_4*x_7+x_2*x_7-x_1*x_6}
	
	
o2 = ideal (x x  - x x  - x x  + x x  + x x  - x x  , x x  - x x  - x x  + x x     
             1 2    1 3    2 7    3 7    2 9    3 10   1 2    1 4    2 7    4 7         
------------------------------------------------------------------------------------------
- x x + x x  , x x  - x x  - x x  + x x  - x x  + x x , x x  - x x - x x  + x x
   4 10  2 11   1 5    4 5    3 7    4 7    1 8    3 8   2 3    2 4   3 6    4 6
------------------------------------------------------------------------------------------
    - x x   + x x  , x x  - x x - x x  + x x  - x x  + x x  , x x  - x x - x x
       4 12    3 13   1 3    1 4   3 7    4 7    4 9    3 11   1 2    1 3   2 7
------------------------------------------------------------------------------------------
    + x x  + x x  - x x   - x x   + x  x  , - x x + x x  , x x   - x x  , - x x  
      3 7    4 9    4 10    9 13    10 14     6 9    5 10   8 12    5 13     8 9   
------------------------------------------------------------------------------------------
+ x x  , - x x   + x x  , - x  x   + x x  , x x  - x x  - x x  + x x  - x x
   5 11     8 10    6 11     11 12    9 13   1 6    4 6    2 7    4 7    1 8
------------------------------------------------------------------------------------------
  + x x , - x x  + x x  + x x  - x x  - x x  + x x , x x  - x x  - x x  + x x  - x x 
    2 8     1 5    2 5    1 6    3 6    2 7    3 7   1 3    1 4    1 5    4 5    4 9
------------------------------------------------------------------------------------------
+ x x  , x x  - x x  - x x  + x x  - x x   + x x  , x x  - x x  - x x  + x x 
   1 14   1 2    1 4    1 6    4 6    4 10    1 13   1 2    1 3    1 6    3 6    
------------------------------------------------------------------------------------------
 - x x   + x x  , - x x  + x x  + x x  - x x  - x x  + x x  ,  - x x  + x x
    3 10    1 12     1 5    4 5    3 7    4 7    8 9    7 14      1 5    4 5
------------------------------------------------------------------------------------------
  + x x  - x x  - x x  + x x  - x x   + x x  , x x  - x x  - x x  + x x 
     1 6    4 6    2 7    3 7    8 12    6 14   2 3    2 4    1 5    4 5
------------------------------------------------------------------------------------------
 + x x  - x x  - x x  + x x  - x x   + x x  , - x x  + x x  + x x  - x x 
    1 6    3 6    2 7    3 7    4 12    2 14     1 6    3 6    2 7    3 7 
------------------------------------------------------------------------------------------
 - x x   + x x  , - x x  + x x  + x x  - x x  - x x   + x x  )
    5 10    7 12     1 6    4 6    2 7    4 7    8 10    7 13

o2 : Ideal of R
	
i3 : isPrime I
	
o3 = true
	
i4 : A=R/I
	
o4 = A
	
o4 : QuotientRing
	
i5 : dim A
	
o5 = 6

\end{lstlisting}
}

\normalsize{We note that by \cref{cor3.4}, $R$ is a graded $\mathbb{C}$-algebra which is an integral domain and generated by $R_1$ elements $X_i,$ $1\le i\le 14.$} 

Let $A=\mathbb{C}[x_{1}, x_{2}, x_{3},..., x_{14}]$ be the polynomial ring in $x_{i}$'s for $1\le i\le 14.$

Let $I$ be the ideal of $A$ generated by the above relations. Then $I$ is a prime ideal of $A.$ Therefore, $A/I$ is an integral domain. Let $\pi: A\longrightarrow A/I$ be the natural surjective homomorphism of $\mathbb{C}$-algebras. Let $\varphi: A\longrightarrow R$ be the $\mathbb{C}$-algebra homomorphism defined by $x_{i}\mapsto X_i.$ By Lemma \ref{lemma6.2}, $\varphi$ induces  $\mathbb{C}$-algebra homomorphism $\tilde{\varphi}:A/I\longrightarrow R.$ 

Note that by \cref{lemma3.1}, $X(w)^{s}_{T}(\mathcal{L}(5\omega_{2}))\neq \emptyset$. Therefore, Krull dimension of $R$ is $6.$ On the other hand, Krull dimension of $A/I$ is $6.$
\begin{lemma}
	$\tilde{\varphi}:A/I\longrightarrow R$ is an isomorphism of $\mathbb{C}$-algebras.
\end{lemma}
\begin{proof}
	Proof is similar to the proof of Lemma \ref{lemma4.6}.  
\end{proof}
\begin{corollary}
	$\tilde{\varphi}:A/I\longrightarrow R$ is an isomorphism as graded $\mathbb{C}$-algebras.
\end{corollary} 
\begin{corollary}
$X$ is not a complete intersection variety.
\end{corollary}

\section{singularities}\label{section7}
In this section, we study the singularity of the torus quotients of Schubert varieties in $G_{2,n}.$ To proceed further, we recall the following proposition on the smooth locus of  Schubert varieties in minuscule Grassmannian due to M. Brion and P. Polo (see \cite[Proposition 3.3(a), p.314]{BP}). We introduce some notation to state the proposition. Let $H$ be a simple algebraic group over algebraically closed field $k.$ Let $Q$ be a minuscule parabolic subgroup of $H.$ Let $W$ (respectively, $W_{Q}$) be the Weyl group of $H$ (respectively, $Q$). Let $W^{Q}$ be the minimal coset representative of  $W/W_{Q}.$ For $w\in W^{Q},$ let $X(w)$ be the Schubert variety in $H/Q$ corresponding to $w.$ Let $X(w)_{sm}$ be the smooth locus of $X(w).$ Then we have   

\begin{proposition}(See \cite[Proposition 3.3(a), p.314]{BP}):\label{proposition7.1}
Let $w\in W^{Q}.$ Let $P_{w}$ be the stabilizer of $X(w)$ in $G.$ Then  $X(w)_{sm}=P_{w}wP/P\subseteq X(w).$
\end{proposition}

We use Proposition \ref{proposition7.1} when $H=SL(n, \mathbb{C}),$ and $Q=P^{\alpha_{2}}.$
Let $w \in W^{Q}$ be such that $w_{min}^{s}\le w$ in $W^{Q}.$ Let $w=(s_{a_{1}}\cdots s_{\frac{n}{2}}\cdots s_{1})(s_{n-1}\cdots s_{2}).$ Then we have $w=(a_{1}+1,n)$ in one line notation, and $a_{1}\ge \frac{n}{2}.$
We observe that $P_{w}$ is a parabolic subgroup of $G$ generated by $B,$ and  $\{s_{i}: i\neq a_{1}+1\}.$ Therefore, by \cref{proposition7.1}, we have $X(w)_{sm}=P_{w}wQ/Q.$ On the other hand, by \cite[Lemma 2.1, p.470]{KP}, $X(w)^{ss}_{T}(\mathcal{L}(\frac{n}{2}\omega_{2}))=\bigcup\limits_{w_{min}^{ss}\le u\le w}(BuQ/Q)^{ss}_{T}(\mathcal{L}(\frac{n}{2}\omega_{2})).$ Note that for $w_{min}^{ss}\le u\le w,$ $u$ is of the form $(s_{a}\cdots s_{\frac{n}{2}-1}\cdots s_{1})(s_{n-1}\cdots s_{2})$ for some $\frac{n}{2}-1\le a\le  a_{1}.$ Therefore, by Proposition \ref{proposition7.1}, we have $X(w)^{ss}_{T}(\mathcal{L}(\frac{n}{2}\omega_{2}))\subseteq X(w)_{sm}.$

Let $X=T\backslash\backslash X(w)^{ss}_{T}(\mathcal{L}(\frac{n}{2}\omega_2)).$ Let $\varphi: X(w)^{ss}_{T}(\mathcal{L}(\frac{n}{2}\omega_2))\longrightarrow X$ be the GIT quotient map. Let $Y=T\backslash X^{s}_{T}(\mathcal{L}(\frac{n}{2}\omega_{2})).$ Since $X(w)^{ss}_{T}(\mathcal{L}(\frac{n}{2}\omega_2))\setminus X(w)^{s}_{T}(\mathcal{L}(\frac{n}{2}\omega_2))$ is a $T$-invariant closed subset of $X(w)^{ss}_{T}(\mathcal{L}(\frac{n}{2}\omega_2)),$ $X\setminus Y$ is a closed subset of $X.$ Hence, $Y$ is an open subset of $X.$ Then we have

\begin{proposition}\label{proposition7.2}
	All the points of $Y$ are smooth points of $X.$
\end{proposition}
\begin{proof}
	We also denote the restriction of the GIT map $\varphi$ to $X(w)^{s}_{T}(\mathcal{L}(\frac{n}{2}\omega_2))$ by $\varphi.$
	Then we have the following commutative diagram:
	\[ \begin{tikzcd}
	X(w)^{s}_{T}(\mathcal{L}(\frac{n}{2}\omega_2)) \arrow{r}{\varphi} \arrow[swap]{d}{\imath} & Y \arrow{d}{\imath} \\%
	X(w)^{ss}_{T}(\mathcal{L}(\frac{n}{2}\omega_2)) \arrow{r}{\varphi}& X
	\end{tikzcd}
	\]
	
	We note that $\varphi: X(w)^{s}_{T}(\mathcal{L}(\frac{n}{2}\omega_2))\longrightarrow Y,$ is a geometric quotient of $X(w)^{s}_{T}(\mathcal{L}(\frac{n}{2}\omega_2))$ by $T.$ Moreover, for each point $x\in X(w)^{s}_{T}(\mathcal{L}(\frac{n}{2}\omega_2)),$ stabilizer $T_{x}$ of $x$ in $T$ is finite. Suppose $x$ is a stable point of $X(w)^{s}_{T}(\mathcal{L}(\frac{n}{2}\omega_2))$ such that $x\in BvQ/Q$ for some $v\le w \in W^{Q}.$ Let $R^{+}(v^{-1})$ denote the set of all positive roots made negative by $v^{-1}.$  Choose a subset $\{\beta_{1},\ldots,\beta_{m}\}$ of $R^{+}(v^{-1})$ such that $x=u_{\beta_{1}}(t_{1})\cdots u_{\beta_{m}}(t_{m})vQ/Q$ with $u_{\beta_{j}}(t_{j})$ in the root subgroup $U_{\beta_{j}},$ $t_{j}\neq0$ for $j=1,..., m.$ The stabilizer $T_{x}$ of $x$ is $\cap_{i=1}^{m} ker(\beta_{j}).$ 
	
Since the stabilizer subgroup $T_{x}$ is finite, the kernel	of the homomorphism $\psi: T\longrightarrow (\mathbb{C}^{\times})^{m}$ defined by $\psi(t)=(\beta_{1}(t),\beta_{2}(t),\ldots, \beta_{m}(t))$ is finite. Let $T'=\phi(T).$ Then $\psi$ induces an injective homomorphism $\tilde{\psi}: X(T')\longrightarrow X(T).$ of character groups such that $\tilde{\psi}(X(T'))$ has a finite index in $X(T).$ Hence, the subset $\{\beta_{1},\ldots,\beta_{m}\}$ of $R^{+}(v^{-1})$ generates the $\mathbb{Q}$-vector space spanned by $S.$ Therefore, by \cite[Lemma 3.2]{K}, $S$ is in the lattice $\sum\limits_{i=1}^{m}\mathbb{Z}\beta_{i}.$ Hence, we have  $T_{x}=\bigcap_{i=1}^{m}ker(\beta_{j})=\bigcap_{i=1}^{n-1}ker(\alpha_{i})=Z(G).$

Let $\pi: G\longrightarrow PSL(n,\mathbb{C})$ be the natural homomorphism. Let $T_{ad}=\pi(T).$  Then $\mathcal{L}(\frac{n}{2}\omega_{2})$ is also $T_{ad}$ linearized. Therefore, $X(w)^{s}_{T}(\mathcal{L}(\frac{n}{2}\omega_{2}))=X(w)^{s}_{T_{ad}}(\mathcal{L}(\frac{n}{2}\omega_{2})).$

 Hence, working with $G_{ad},$ we may assume $T_{x}=\{e\}.$ Thus $\varphi: X(w)^{s}_{T}(\mathcal{L}(\frac{n}{2}\omega_2)) \longrightarrow Y,$ is a principal $T$-bundle. Therefore, $Y$ is smooth as $X(w)^{s}_{T}(\mathcal{L}(\frac{n}{2}\omega_2))$ is smooth. Since $Y$ is an open subset of $X,$  $Y$ is contained in the smooth locus of $X.$
\end{proof}

\begin{proposition}\label{prop 3.3}
Let $\xi \in X(w_{min}^{ss})^{ss}_{T}(\mathcal{L}(\frac{n}{2}\omega_{2}))\subseteq X(w)^{ss}_{T}(\mathcal{L}(\frac{n}{2}\omega_{2})) .$ Let $K=\{v\in W^{P^{\alpha_{\frac{n}{2}}}}: v\xi\in \bigcup\limits_{w_{min}^{ss}\le u\leq w} BuQ/Q\},$ and
$L=\{v\varphi(\xi): v\in K\}.$ Then the set of all singular points of $X$ is a subset of $L.$
\end{proposition}
\begin{proof}
Let $X_{\text{Sing}}$ be the set of all singular points of $X.$ Then by \cref{proposition7.2}, $X_{\text{Sing}}$ is a subset of $X\setminus Y.$ 

We observe that for $v\in W_{P^{\alpha_{\frac{n}{2}}}},$  $v(Bw_{min}^{ss}Q/Q)^{ss}_{T}(\mathcal{L}(\frac{n}{2}\omega_{2}))=(Bw_{min}^{ss}Q/Q)^{ss}_{T}(\mathcal{L}(\frac{n}{2}\omega_{2})).$ Therefore, by \cite[Lemma 2.1, p.470]{KP}, $X(w)^{ss}_{T}(\mathcal{L}(\frac{n}{2}\omega_{2}))\setminus X(w)^{s}_{T}(\mathcal{L}(\frac{n}{2}\omega_{2}))=\bigcup\limits_{v\in K}v(Bw_{min}^{ss}Q/Q)^{ss}_{T}(\mathcal{L}(\frac{n}{2}\omega_{2})).$ Since $\xi \in X(w_{min}^{ss})^{ss}_{T}(\mathcal{L}(\frac{n}{2}\omega_{2}))$ and there is a unique $T$-invariant standard monomial in degree one whose restriction to $Bw_{min}^{ss}Q/Q$ is non-zero, we have $\varphi((Bw_{min}^{ss}Q/Q)^{ss}_{T}(\mathcal{L}(\frac{n}{2}\omega_{2})))=\varphi(\xi).$ Further, we note that $\varphi(v\xi)=v\varphi(\xi).$ Therefore, we have  $X_{\text{Sing}}\subseteq X\setminus Y=L.$ 
\end{proof}
\begin{corollary}\label{cor7.4}
	$X_{\text{Sing}}$ is finite.
\end{corollary}
\begin{proof}
follows from Proposition \ref{prop 3.3}.
\end{proof}

We have already observed that $X$ is smooth when $w=w_{min}^{s}$ (see Proposition \ref{lemma3.5}). Thus, we are interested in studying the singularity of $X,$ when $w^{s}_{min}< w.$ 

Let $w=(s_b\cdots s_{\frac{n}{2}}\cdots s_1)(s_{n-1}\cdots s_2).$ Then in one line notation we have $w=(b+1, n).$ Note that $b \geq \frac{n}{2}+1,$ as $w_{min}^{s}<w.$ Let $\xi, \varphi$ be as defined above and  $K,L$ be as defined in Proposition \ref{prop 3.3}.  By Corollary \ref{cor7.4}, $X_{Sing}$ is finite and $X_{Sing}\subseteq L.$  Next we prove that $L=X_{Sing}.$ Note that to prove $L=X_{Sing},$ it is enough to show that $\varphi(\xi)$ is a singular point of $X.$

To proceed further we first introduce some definition and notation following the article by Jean-Marc Dr\'ezet (see \cite[p. 28--29]{Dre}).

Let $H$ be a reductive group acting on an algebraic variety $Z.$ Given a character $\chi$ of $H,$ we define the $H$-line bundle $\mathcal{L}(\chi)$ associated to $\chi$ as follows:
The underlying line bundle is $\mathcal{O}_{Z}=Z\times \mathbb{C}$ and the action of $H$ is 
\begin{center}
$H\times Z\times \mathbb{C}\longrightarrow Z\times\mathbb{C} $

$(h,z,t)\longmapsto (hz, \chi(h)t).$
\end{center}

Let $x=\begin{pmatrix}\\
&x_{1 \frac{n}{2}} & 0\\
&x_{2 \frac{n}{2}} & 0\\
&\vdots & \vdots\\
&x_{\frac{n}{2}-1 \frac{n}{2}} & 0\\
&1 & x_{\frac{n}{2}+1 n}\\
&0 & x_{\frac{n}{2}+2 n}\\
&\vdots & \vdots \\
&0 & x_{n-1 n}\\
&0 & 1
\end{pmatrix} \in Bw^{ss}_{min}Q/Q$ be such that $(\prod\limits_{i=1}^{\frac{n}{2}-1}x_{i,\frac{n}{2}}) (\prod\limits_{i=\frac{n}{2}+1}^{n-1}x_{i,n})\neq 0.$

Since $\prod\limits_{i=1}^{\frac{n}{2}} p_{i\frac{n}{2}+i}$ is the unique monomial (up to a non zero scalar) of $H^0(X(w_{min}^{ss}),\mathcal{L}(\frac{n}{2}\omega_{2}))^{T}$ such that $\prod\limits_{i=1}^{\frac{n}{2}} p_{i\frac{n}{2}+i}(x)\neq 0,$ $x \in X(w_{min}^{ss})^{ss}_{T}(\mathcal{L}(\frac{n}{2}\omega_{2})).$ Since $\varphi((Bw^{ss}_{min}Q/Q)^{ss})=\varphi(\xi),$ we have $x\in \varphi^{-1}(\varphi(\xi)).$ The stabilizer $T_{x}$ of $x$ in $T$  is $ \bigcap\limits_{i\neq \frac{n}{2}} ker (\alpha_{i}),$ that is $T_{x}=\lambda_{\frac{n}{2}},$ where $\lambda_{\frac{n}{2}}=\text{diag}(t^{\frac{n}{2}},\cdots, t^{\frac{n}{2}},t^{-\frac{n}{2}},\cdots,t^{-\frac{n}{2}}).$ Furthermore, since $Tx$ is a minimal dimensional orbit in $\varphi^{-1}(\varphi(\xi)),$  $Tx$ is closed in $\varphi^{-1}(\varphi(\xi)).$ Then we prove the following

\begin{proposition}\label{prop 7.5}
$\varphi(\xi)$ is a singular point of $X.$ 
\end{proposition}
\begin{proof}
Recall that $X(w)^{ss} \subseteq  X(w)_{sm}.$ Further, $X(w)^{ss} \setminus X(w)^{s}=\bigcup\limits_{ v\in K  }v(Bw_{min}^{ss}Q/Q)^{ss}.$ Thus the  co-dimension of $X(w)^{ss} \backslash X(w)^{s}$ in $X(w)^{ss}$ is $b+1-\frac{n}{2}$ which is at least $2,$ as $b \geq \frac{n}{2}+1.$ Further $\varphi|_{X(w)^{s}}: X(w)^{s} \longrightarrow \varphi(X(w)^{s})$ is a geometric quotient and $T$ acts freely on $X(w)^{s}.$  

By the above discussion $x\in \varphi^{-1}(\varphi(\xi)),$ and $Tx$ is a close orbit in $\varphi^{-1}(\varphi(\xi))$ with $T_{x}=\lambda_{\frac{n}{2}}.$ Take $\chi=\alpha_{\frac{n}{2}}.$ Then $\chi$ is non trivial character on $\lambda_{\frac{n}{2}},$ i.e., $\chi(\lambda_{\frac{n}{2}})\neq {1}.$ Therefore, by \cite[Theorem 8.3, p.44]{Dre}, $\mathcal{O}_{X,\phi(\xi)}$ is not a unique factorization domain. Hence, $\varphi(\xi)$ is a singular point of $X.$ 
\end{proof}
\begin{corollary}
$L$ is precisely the set of singular points of $X.$
\end{corollary}
\begin{proof}
Follows from Proposition \ref{prop 3.3} and Proposition \ref{prop 7.5}.
\end{proof}
Next we compute the number elements in $L$ for $T\backslash \backslash (G_{2,n})^{ss}_{T}(\mathcal{L}(\frac{n}{2}\omega_{2})).$  
\begin{lemma}\label{lemma3.2}
Let $x_{1},x_{2},y_{1},y_{2}$ be positive integers such that $x_{1}\neq x_{2},$ $y_{1}\neq y_{2}.$ If $x_1 < y_1$ and $x_2 < y_2,$ then $(x_1, x_2)\uparrow < (y_1, y_2) \uparrow$.
\end{lemma} 

\begin{proof}
It is enough to prove that min$\{x_1, x_2\}$ $<$ min$\{y_1, y_2\}$ and max$\{x_1, x_2\}$ $<$ max$\{y_1, y_2\}$.
	
Let min$\{x_1, x_2\} = x_1$. Then $x_1 < x_2$. We have $x_1 < y_1$ and $x_1 < x_2 < y_2$. So, $x_1 < y_1$ and $x_1 < y_2$.  Hence, $x_1 < $ min$\{y_1,y_2\}$. Similarly, if min$\{x_1, x_2\} = x_2$ then $x_2 < $ min$\{y_1,y_2\}$. Hence, min$\{x_1, x_2\}$ $<$ min$\{y_1, y_2\}$.
	
Let max$\{y_1, y_2\} = y_1$. Then $y_2 < y_1$. We have $x_1 < y_1$ and $x_2 < y_2 < y_1$. So, $x_1 < y_1$ and $x_2 < y_1$.  Hence, max$\{x_1,x_2\}$ $< y_1$. Similarly, if max$\{y_1, y_2\} = y_2$ then max$\{x_1,x_2\}$ $< y_2$. Hence, max$\{x_1, x_2\}$ $<$ max$\{y_1, y_2\}$.
	
Therefore, $(x_1, x_2)\uparrow < (y_1, y_2) \uparrow$.
\end{proof}
Let $R_{1}=H^0(G_{2,n}, \mathcal{L}(\frac{n}{2}\omega_{2}))^{T}$ and  dim($R_{1}$)=$d.$ Let $\{X_{i}: 1\le i\le d\}$ be the standard monomial basis of $R_{1}$ where $X_{1}=\prod\limits_{i=1}^{\frac{n}{2}} p_{i\frac{n}{2}+i}.$  
\begin{lemma}\label{lem 3.6}
Let $\xi$ be as above. Then given $w \in W^{S \backslash \{\alpha_\frac{n}{2}\}}$ such that $w \neq id$, and $ w \neq w_0^{S \backslash \{\alpha_{\frac{n}{2}}\}}$ there exists an $i \neq 1$ such that $X_i(w\xi) \neq 0.$
\end{lemma}
\begin{proof}
Assume that $w \in W^{S \backslash \{\alpha_\frac{n}{2}\}}$ be such that $w \neq id$ and $ w \neq w_0^{S \backslash \{\alpha_{\frac{n}{2}}\}}.$
Note that $w(1) < w(2) < \cdots < w(\frac{n}{2})$ and $w(\frac{n}{2}+1) < w(\frac{n}{2}+2) < \cdots < w(n).$ Further, since $w\neq id, w_{0}^{S\setminus\{\alpha_{\frac{n}{2}}\}},$ we have $(w(\frac{n}{2}), w(n))\uparrow\neq (\frac{n}{2}, n).$ Indeed, if $(w(\frac{n}{2}), w(n))\uparrow= (\frac{n}{2}, n),$ then either $w(\frac{n}{2})=\frac{n}{2}$ or  $w(\frac{n}{2})=n.$ Now, $w=id$ or $w=w_{0}^{S\setminus\{\alpha_{\frac{n}{2}}\}}$ according as $w(\frac{n}{2})=\frac{n}{2}$ or $w(\frac{n}{2})=n.$

Let $M=p_{(w(1),w(\frac{n}{2}+1))\uparrow}p_{(w(2),w(\frac{n}{2}+2))\uparrow}\cdots p_{(w(\frac{n}{2}),w(n))\uparrow}.$	By using \cref{lemma3.2}, we see that $M$ is a standard monomial. Then $M=X_{i}$ where $i\neq 1.$

Now we prove that $M(w\xi) \neq 0$. Note that $w(1)$-th row of $w(\xi)=$ $1$-st row of $\xi$ and $w(\frac{n}{2}+1)$-th row of $w(\xi)=$ $\frac{n}{2}+1$-th row of $\xi$. Hence, $p_{(w(1),w(\frac{n}{2}+1))\uparrow}(w\xi) = \pm p_{1,\frac{n}{2}+1}(\xi)$. Similarly, $p_{(w(i),w(\frac{n}{2}+i))\uparrow}(w\xi) = \pm p_{i,\frac{n}{2}+i}(\xi)$. Therefore, $M(w\xi)=\pm \prod_{i=1}^{\frac{n}{2}}p_{i,\frac{n}{2}+i}(\xi) \neq 0$   
\end{proof}

Note that $(w_{0}^{S\setminus \{\alpha_{n/2}\}})^{-1}=w_{0}^{S\setminus \{\alpha_{n/2}\}}$ and $w_{0}^{S\setminus \{\alpha_{n/2}\}}\xi\in Bw_{min}^{ss}Q/Q.$ Thus, $\varphi(w_{0}^{S\setminus \{\alpha_{n/2}\}}\xi)=\varphi(\xi).$ On the other hand, we have the following 
\begin{lemma}\label{lemma 3.5}
Let $\xi,$ $\varphi$  be as defined above. If $\varphi(u\xi)=\varphi(v\xi)$ for some $u, v\in W^{S\setminus\{\alpha_{n/2}\}},$ then $v=u$ or $v=uw_{0}^{S\setminus \{\alpha_{n/2}\}}.$ 
\end{lemma}
\begin{proof}
Assume that $\varphi(u\xi)=\varphi(v\xi).$ Since $\varphi$ is $W$-equivariant, we have $\varphi(u^{-1}v\xi)=\varphi(\xi).$ Let $u^{-1}v=u_{1}v_{1}$ where $u_{1}\in W^{S\setminus\{\alpha_{\frac{n}{2}}\}}$ and $v_{1}\in W_{S\setminus\{\alpha_{\frac{n}{2}}\}}.$ Since $v_{1}\in W_{S\setminus\{\alpha_{\frac{n}{2}}\}},$ $v_{1}\xi \in Bw_{min}^{ss}Q/Q.$ Thus, we have $\varphi(u_{1}\xi)=\varphi(\xi).$ Then, by 
\cref{lem 3.6}, we have $u_{1}=id$ or $u_{1}=w_{0}^{S\setminus \{\alpha_{n/2}\}}.$ Thus, we have  $v=u$ $mod~W_{S\setminus\{\alpha_{\frac{n}{2}}\}}$ or $v=uw_{0}^{S\setminus \{\alpha_{n/2}\}}$ $mod~ W_{S\setminus\{\alpha_{\frac{n}{2}}\}}.$ Further, note that
$-w_{0}(S\setminus\{\alpha_{\frac{n}{2}}\})=S\setminus\{\alpha_{ \frac{n}{2}}\}.$ Hence, $uw^{S\setminus\{\alpha_{n/2}\}}(\alpha_{j})$ is a positive root for $j\neq \frac{n}{2},$ which implies   $uw_{0}^{S\setminus \{\alpha_{n/2}\}}\in W^{S\setminus\{\alpha_{\frac{n}{2}}\}}.$ Therefore, proof follows.
\end{proof}
\begin{corollary}\label{cor7.10}Let $L$ be as defined above. Let Card$(W^{S\setminus\{\alpha_{\frac{n}{2}}\}})=A.$ Then 
Card$(L)$=$\frac{A}{2},$ i.e., $T\backslash\backslash G_{2,n}^{ss}(\frac{n}{2}\omega_{2})$ has $\frac{A}{2}$ singular points. 
\end{corollary}
\begin{proof}
Let $c: W^{S\setminus\{\alpha_{\frac{n}{2}}\}}\longrightarrow L$ be map defined by $c(u)=\varphi(u\xi)$ for $u\in W^{S\setminus\{\alpha_{\frac{n}{2}}\}}.$ Clearly, $c$ is surjective. On the other hand, by Lemma \ref{lemma 3.5}, $c^{-1}(\varphi(u\xi))=\{u, uw_{0}^{S\setminus \{\alpha_{\frac{n}{2}}\}} \}$ for $\varphi(u\xi)\in L.$ Therefore, Card$(L)$=$\frac{A}{2}.$ 
\end{proof}

\begin{example}{ $T\backslash\backslash(G_{2,6})^{ss}_{T}(\mathcal{L}(3\omega_2))$ is not smooth:}
 Let $X=T\backslash\backslash(G_{2,6})^{ss}_{T}(\mathcal{L}(3\omega_2)).$ 
	Recall that $X$ is a cubic hypersurface in the projective space $\mathbb{P}^{4}=Proj(\mathbb{C}[X_{1},X_{2},X_{3},X_{4},X_{5}]),$ given by   $F(X_{1},X_2,X_3,X_4,X_5)=X_3X_4^{2}-X_{1}X_{2}X_{5}+X_1X_3X_4-X_2X_3X_4+ X_2X_3X_5-X_3X_4X_5.$ Let $\xi \in (Bw_{min}^{ss}P^{\alpha_{2}}/P^{\alpha_{2}})^{ss}_{T}(\mathcal{L}(3\omega_2)).$ Then we have $\varphi((Bw_{min}^{ss}P^{\alpha_{2}}/P^{\alpha_{2}})^{ss}_{T}(\mathcal{L}(3\omega_2)))=\varphi(\xi)=p$ where  $p=[1:0:0:0:0].$ Let $\tilde{p}=(1,0,0,0,0)$ be a lift of $p$ in the affine cone of $X.$
	
	Then we have
	$\frac{\partial F}{\partial X_{1}}= X_{3}X_{4}-X_{2}X_{5};
	\hspace{.1cm}\frac{\partial F}{\partial X_{2}}=X_{3}X_{5}-X_{1}X_{5}-X_{3}X_{4};
	\hspace{.1cm} \frac{\partial F}{\partial X_{3}}=X_{4}^{2}+X_{1}X_{4}-X_{2}X_{4}+X_{2}X_{5}-X_{4}X_{5};
	\hspace{.1cm}\frac{\partial F}{\partial X_{4}}=2X_{3}X_{4}+X_{1}X_{3}-X_{2}X_{3}-X_{3}X_{5};
	\hspace{.1cm}\frac{\partial F}{\partial X_{5}}=-X_{1}X_{2}+X_{2}X_{3}-X_{3}X_{4}.$
	
	Thus $(\frac{\partial F}{\partial X_{1}}(\tilde{p})~~\frac{\partial F}{\partial X_{2}}(\tilde{p})~~\frac{\partial F}{\partial X_{3}}(\tilde{p})~~\frac{\partial F}{\partial X_{4}}(\tilde{p})~~\frac{\partial F}{\partial X_{5}}(\tilde{p}))=(0~~0~~0~~0~~0).$ Hence, $\tilde{p}$ is a singular point of the affine cone of $X.$ Therefore, $p$ is a singular point of $X.$ 
\end{example}

\begin{example}{ $T\backslash\backslash(X(6,8))^{ss}_{T}(\mathcal{L}(4\omega_2))$ is not smooth:}
Let $X=T\backslash\backslash(X(6,8))^{ss}_{T}(\mathcal{L}(4\omega_2)).$ \\
Recall that $X$ is a projective subvariety of the projective space\\ $\mathbb{P}^{8}=Proj(\mathbb{C}[X_{1},X_{2},X_{3},X_{4},X_{5},X_{6},X_{7},X_{8},X_{9}]),$ given by the following system of quadratic homogeneous polynomials:
	
	$F_{1}=X_2X_6-X_5X_8+X_4X_8-X_1X_8-X_2X_4+X_1X_2$
	
	$F_{2}=X_1X_3-X_2X_4+X_3X_6-X_3X_8+X_4X_8-X_4X_9$
	
	$F_{3}=X_2X_7-X_1X_7-X_3X_6+X_4X_9$
	
	$F_{4}=X_3X_6-X_5X_7$
	
	$F_{5}=X_1X_3-X_2X_4+X_2X_6-X_3X_8+X_4X_8-X_1X_9$.

Note that $\varphi((Bw_{min}^{ss}P^{\alpha_{2}}/P^{\alpha_{2}})^{ss}_{T}(\mathcal{L}(4\omega_2)))=\varphi(\xi)=p$ where  $p=[0:0:0:0:0:0:0:0:1].$ Let $\tilde{p}=(0,0,0,0,0,0,0,0,1)$ be a lift of $p$ in the affine cone of $X.$

\begin{equation*}
\tiny{\begin{pmatrix}
\frac{\partial F_{i}}{\partial X_{j}}
\end{pmatrix}=
\begin{pmatrix}
	  \\
	  &X_{2}-X_{8} & X_{6}-X_{4}+X_{1} & 0 & X_{8}-X_{2}& -X_{8}& X_{2}&0&X_{4}-X_{5}-X_{1}&0 \\
	  &X_{3} & -X_{4} & X_{1}+X_{6}-X_{8} & X_{8}-X_{2}-X_{9}& 0& X_{3}&0&X_{4}-X_{3}&-X_{4}\\
	   &-X_{7} & X_{7} & -X_{6} & X_{9}& 0& -X_{3}&X_{2}-X_{1}&0&X_{4}\\
	  &0 & 0 & X_{6} & 0& -X_{7}& X_{3}&-X_{5}&0&0\\
	   &X_{3}-X_{9} & X_{6}-X_{4} & X_{1}-X_{8} & X_{8}-X_{2}& 0& X_{2}&0&X_{4}-X_{3}&-X_{1}\\
\end{pmatrix}}
\end{equation*}
	
	\normalsize{Thus we have} \begin{equation*}
		\tiny{\begin{pmatrix}
			\frac{\partial F_{i}}{\partial X_{j}}
			\end{pmatrix}=
			\begin{pmatrix}
				\\
			&0 &0 &0 &0 &0 &0 &0 &0 &0 \\
			&0 &0 &0 &-1&0 &0 &0 &0 &0\\
			&0 &0 &0 &1 &0 &0 &0 &0 &0 \\
			&0 &0 &0 &0 &0 &0 &0 &0 &0\\
			&-1&0 &0 &0 &0 &0 &0 &0 &0\\
\end{pmatrix}}
\end{equation*}
\normalsize{Hence, $\tilde{p}$ is a singular point of the affine cone of $X,$ as the rank of $\begin{pmatrix}\frac{\partial F_{i}}{\partial X_{j}}(\tilde{p})\end{pmatrix}<9-5=4.$ Therefore, $p$ is a singular point of $X.$}

\end{example}

\end{document}